\documentclass[aap]{imsart}

\RequirePackage{amsthm,amsmath,amsfonts,amssymb}
\usepackage{hyperref}
\usepackage{mathrsfs}
\usepackage{euscript}
\RequirePackage[numbers]{natbib}
\usepackage{subcaption}
\RequirePackage{graphicx}

\startlocaldefs

\theoremstyle{plain}
\newtheorem{thm}{Theorem}

\newtheorem{cor}[thm]{Corollary}
\newtheorem{lem}[thm]{Lemma}
\newtheorem{prop}[thm]{Proposition}
\newtheorem{claim}[thm]{Claim}

\theoremstyle{remark}
\newtheorem{definition}{Definition}[section]

\newtheorem{remark}{Remark}[section]

\newcommand{\ccc}{\mathbf{c}}

\newcommand{\D}{\mathbb{D}}
\newcommand{\ddd}{\operatorname{d}}
\newcommand{\E}{\mathbb{E}}

\newcommand{\GG}{\mathscr{G}}

\newcommand{\hp}{\mathbf{h}}

\newcommand{\m}{\mathfrak{m}}
\newcommand{\M}{\mathcal{M}}

\newcommand{\N}{\mathbb{N}}
\renewcommand{\P}{\mathbb{P}}
\newcommand{\PP}{\mathscr{P}}
\newcommand{\QQ}{\mathscr{Q}}

\newcommand{\R}{\mathbb{R}}

\newcommand{\T}{\mathscr{T}}

\newcommand{\ww}{\mathtt{w}}

\newcommand{\XX}{\mathfrak{X}}

\newcommand{\z}{\mathbf{z}}

\newcommand{\BF}{{\operatorname{BF}}}
\newcommand{\DF}{{\operatorname{DF}}}
\newcommand{\dd}{\mathbf{d}}

\newcommand{\dist}{\operatorname{{dist}}}
\newcommand{\e}{\mathbf{e}}
\newcommand{\ee}{\tilde{\e}}
\newcommand{\ER}{Erd\H{o}s-R\'{e}nyi }
\newcommand{\eps}{\varepsilon}
\newcommand{\fl}[1]{\lfloor #1 \rfloor}
\newcommand{\GHP}{{\operatorname{GHP}}}
\newcommand{\ghp}{{Gromov-Hausdorff-Prokhorov }}

\newcommand{\Leb}{\operatorname{Leb}}

\newcommand{\scale}{\operatorname{scale}}

\newcommand{\sur}{\operatorname{{sur}}}

\newcommand{\weakarrow}{\overset{d}{\Longrightarrow}}

\endlocaldefs

\begin{document}
	
	\begin{frontmatter}
		\title{Epidemics on critical random graphs with heavy-tailed degree distribution}
		\runtitle{Epidemics on critical graphs}
		
		\begin{aug}
			\author[A]{\fnms{David} \snm{Clancy, Jr. }\ead[label=e1]{djclancy@uw.edu}},
			\address[A]{University of Washington, Department of Mathematics, \printead{e1}}
			
		\end{aug}
		
		\begin{abstract}
			We study the susceptible-infected-recovered (SIR) epidemic on a random graph chosen uniformly over all graphs with certain critical, heavy-tailed degree distributions. For this model, each vertex infects all its susceptible neighbors and recovers the day after it was infected. When a single individual is initially infected, the total proportion of individuals who are eventually infected approaches zero as the size of the graph grows towards infinity. Using different scaling, we prove process level scaling limits for the number of individuals infected on day $h$ on the largest connected components of the graph. The scaling limits are contain non-negative jumps corresponding to some vertices of large degree, that is these vertices are super-spreaders. Using weak convergence techniques, we can describe the height profile of the $\alpha$-stable continuum random graph \cite{GHS.18, CKG.20}, extending results known in the Brownian case \cite{MS.19}. We also prove abstract results that can be used on other critical random graph models. 
		\end{abstract}
		
		
		
	\end{frontmatter}
	

	\section{Introduction}
	
	Consider the following simple susceptible-infected-recovered (SIR) model of disease spread in discrete time. On day $0$, a single individual becomes infected with a disease. On day 1, that single infected individual comes into contact with some random number (possibly zero) of non-infected individuals and transmits the disease. After transmitting the disease to others, this initial infected individual is cured and can never catch the disease again. On subsequent days each infected individual does the same thing: they come into contact with some non-infected individuals, transmit the disease but then are cured. The study of how the disease spreads over time naturally gives rise to a graph \cite{BM.90} constructed in a breadth-first order, see Figure \ref{fig:examplepicture1} for an example of a small outbreak and Figure \ref{fig:bigoutbreak1} for an example of a larger outbreak. The individuals are represented by vertices, and an edge between two vertices represents that a vertex closer to the source transmitted the disease to the other. Knowing the graph and the source tells us more information than the number of individuals infected on a particular day, it tells us the history of how the disease spread from individual from individual.
	
	The size of the outbreak then corresponds to the size of a connected component in the graph and, more importantly for our work, the number of people infected on day $h = 0,1,\dotsm$ is just the number of vertices at distance $h$ from a root vertex corresponding to the initially infected individual. Let $Z_n(h)$ represent the number of people infected on day $h\ge 0$ when the total population is of size $n$. The process $Z_n(h)$ is just the \textit{height profile} of the component containing the initially infected individual. We are interested in the describing $n\to\infty$ scaling limits of $Z_{n}(h)$ for the macroscopic outbreaks for certain critical random graphs which exhibit a ``super-spreader'' phenomena - that is they possess vertices with large degree.
	
	\begin{figure}
		\centering
		\includegraphics[width=0.5\linewidth]{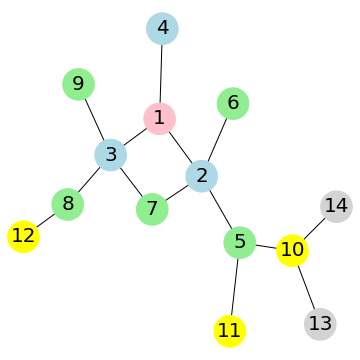}
		\caption{A small outbreak. Here, on day 0 the vertex labeled 1 is infected. The vertex 1 transmits the disease to vertices 2, 3 and 4 (in blue) who become the infected population on day 1. The vertices infected on day 1 will infect the green vertices (5 through 9) who are infected on day 2. This continues with the yellow vertices becoming infected on day 3, and the grey vertices on day 4. }
		\label{fig:examplepicture1}
	\end{figure}

	A classical probabilistic model in this area is the so-called Reed-Frost model, where each individual comes into contact with every non-infected individual independently with probability $p$. It is not hard to see that the corresponding graph is the \ER random graph $G(n,p)$ where each edge is independently added with probability $p$. This object is well-studied, and we know that in the critical window $p = p(n) = n^{-1}+\lambda n^{-4/3}$ the size of the macroscopic outbreaks are of order $n^{2/3}$ \cite{Aldous.97}. Within this critical window each vertex has approximately Poisson(1) many neighbors, so in particular it has light tails. In turn, the process $Z_n(h)$ corresponding to the largest component has a scaling limit and that limit is a continuous process \cite{MS.19}. We stress that this is not because we are looking only at an epidemic started from a single individual. The same can be said if we infect $O(n^{1/3})$ individuals on day $0$ \cite{Clancy.20}.
	
	To capture some super-spreading phenomena we focus mostly on the configuration model with a heavy-tailed degree distribution: $\P(\deg(i) = k) \sim c k^{-(2+\alpha)}$ for some $\alpha\in(1,2)$, along with some other technical assumptions dealing with criticality. The configuration model is a graph on $n$ vertices chosen randomly over all graphs with a prescribed degree sequence. See Chapter 7 of \cite{vanderHofstad.17} for an introduction to this model. We omit the case $\alpha = 2$ because this model falls within the same universality class as the critical \ER random graph $G(n,n^{-1}+\lambda n^{-4/3})$ \cite{BBSW.14,CKG.20} and so, up to some scaling factors, the structure of the processes $Z_{n}(h)$ on largest components (which correspond to the largest possible outbreaks) will be asymptotically the same as those in the \ER random graph. In the asymptotic regime we study, the largest outbreaks are of order $O(n^{\frac{\alpha}{\alpha+1}})$ and scaling limits of $Z_{n}(h)$ will possess positive jumps. These positive jumps come from presence of the super-spreading individuals.
	
	We also restrict our focus to critical regimes. One reason is general principle that what happens at a phase transition is often interesting. Another is that while there are some important results on the structure of the largest components of the critical heavy-tailed configuration model \cite{CKG.20,Joseph.14}, there is not much information on the structure of the disease outbreaks. In this vein, there are results in the literature on the behavior of the largest outbreak when initially only a single individual is infected. While studying a model similar to ours where edges are kept with probability $p\in[0,1]$ but are otherwise deleted, the authors of \cite{BJML.07} show that there is a parameter $R_0$ such that if $R_0\le 1$ then only outbreaks of size $o(n)$ as $n\to\infty$ can occur whereas if $R_0> 1$ there is a positive probability that an outbreak of size $O(n)$ occurs as $n\to\infty$. See also \cite{MR.95,MR.98,JL.09}. A continuous time analog of that model was studied in \cite{BP.12} and there the authors show that there is a similar phase transition between outbreaks of size $o(n)$ and outbreaks which are of size $O(n)$ with positive probability. Those authors also describe some of the large $n$ behavior of $Z_n(t)$ (the number of individuals infected at a continuous time $t \ge 0$) conditionally on having an outbreak of size $O(n)$, but they do not provide information for what happens at the phase transition. We hope to fill in this gap in the literature.

	\begin{figure}
		\centering
		\includegraphics[width=0.9\linewidth]{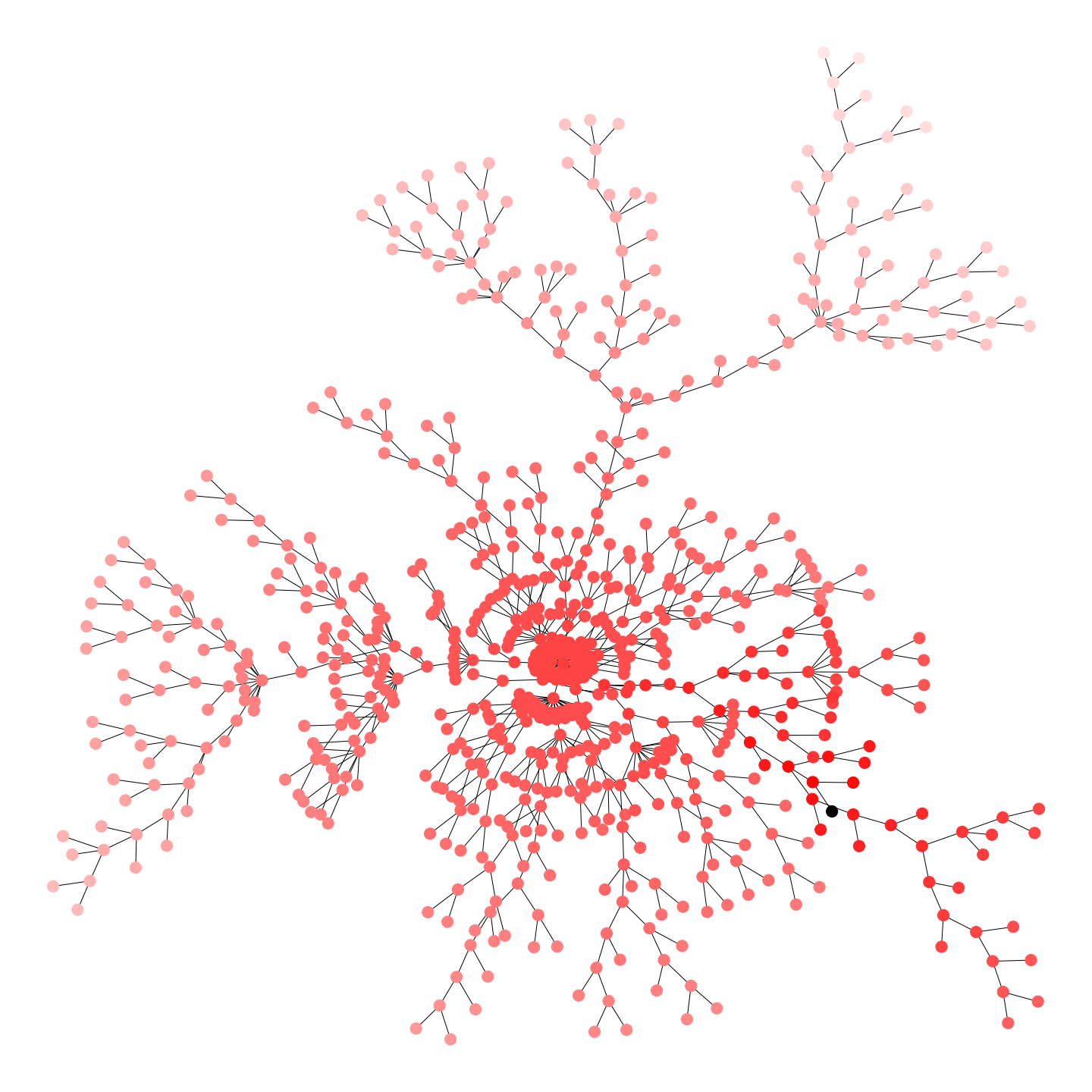}
		\caption{A simulation of the largest outbreak on a configuration model with heavy-tailed degree distribution with $\alpha = 3/2$. This component has 735 vertices, while the entire graph has 70,000. The black node is the first vertex to be infected, and then darker shades indicate that the corresponding vertex infected earlier in the outbreak. Most of the vertices have small degree ($\le 3$); however, there are some vertices with large degree.  The large red blob in the middle of the image comes from a vertex of relatively large degree, i.e. a super-spreader. We can also see that there is another super-spreader depicted just below that red blob.}
		\label{fig:bigoutbreak1}
	\end{figure}
	
	\subsection{Weak convergence results}

	Let us discuss a little more formally the configuration model. Before doing so, we recall that a multi-graph can have multiple edges and self-loops while a simple graph does not contain multiple edges nor self-loops. In terms of our approach to studying epidemics, self-loops and multiple edges do not make any physical sense because, for example, an infected individual cannot reinfect themself. 
	
	Given $\dd^{n} = (d_1,\dotsm,d_n)$ a finite sequence of strictly positive integers $d_j\ge 1$, the configuration model $M(\dd^n)$ is the random multi-graph chosen randomly over all multi-graphs $G$ on the vertex set $[n]:=\{1,\dotsm,n\}$ where the degree (counted with multiplicity) of vertex $j$ is $\deg(j) = d_j$. In order to construct such a multi-graph we need $\sum_{j=1}^n d_j$ to be even, and two algorithms for its construction will be discussed in Section \ref{sec:configuration}. We say that any such graph $G$ has degree sequence $\dd^n$.
	
	A priori it may not be possible to construct a simple graph on with degree sequence $\dd^n$ because, for example, a single vertex may have degree $d_i> \sum_{j\neq i} d_j$. However, if there is a simple graph with degree sequence $\dd^n$, then conditionally on the event $\{M(\dd^n) \text{ is simple}\}$ the graph is uniformly distribution over all simple graphs with degree sequence $\dd^n$ \cite[Proposition 7.15]{vanderHofstad.17}. Moreover for the asymptotic regime we study it makes no difference \cite{CKG.20} whether or not we examine simple graphs or multi-graphs so we will just say ``graph.''
	
	One aspect of randomness for the configuration model comes from taking the graph to be randomly constructed over all graphs with a fixed deterministic degree sequence. Another comes from taking the degree sequence itself to be random, say, with a common distribution $\nu$ on $\{1,2,\dotsm\}$. We then generate the graph conditionally given this degree distribution. That is we generate $M(\dd^n)$ where $d_j$ are i.i.d. with common law $\nu$. We may have to replace $d_n$ with $d_n+1$ to obtain the proper parity; however, this does not affect the analysis \cite{CKG.20}. To distinguish between these two situations we will write $M_n(\nu)$ instead of $M(\dd^n)$. 
	
	We focus on the degree distributions studied by Joseph \cite{Joseph.14} and Conchon-Kerjan and Goldschmidt \cite{CKG.20}:
	\begin{equation}
		\label{eqn:cdelta1}
		\lim_{k\to\infty}  k^{(\alpha+2)}\nu(k) = c\in (0,\infty),\qquad \E[d_1] = \delta\in (1,2),\qquad \E[d_1^2] = 2\E[d_1],
	\end{equation} for some $\alpha\in (1,2)$. The third statement about the second moment and the mean imply that we are examining the random graph at criticality \cite{MR.95,MR.98,JL.09}. This means that there is no giant component, i.e. there is no single component which contains a positive proportion of the total number of vertices. Instead, there are macroscopic components which are of order $O(n^{\frac{\alpha}{\alpha+1}})$. 
	
	In order to obtain scaling limits for a height profile represent the number of people infected on day $h$, we would either need to look at the case where a significant number of individuals are infected on day zero, or focus on the largest possible outbreaks. We focus on the latter situation and hence decompose the graph $M_n(\nu)$ into its connected components $G_n^1, G_n^2,\dotsm$ where they are indexed so
	\begin{equation*}
		\#G_n^1\ge \#G_n^2\ge \dotsm.
	\end{equation*} In order to know how a disease spreads through $G_n^i$, we need to know its source. We will start the spread from a single vertex $\rho_n^i$ chosen with probability proportional to its degree, and we will say that the component $G_n^i$ is \textit{rooted} at the vertex $\rho_n^i$. 
	
	The selection of $\rho_n^i$ is a size-biased sample and not a uniform sample, but this is for good reason. In terms of how a disease spreads through a community, vertices with higher degree have more neighbors from whom they can catch the disease and so we should expect these vertices to be infected earlier in the outbreak. This has been observed in a survey of how influenza (seasonal or the H1N1 variant) spread through Harvard in 2009 \cite{CF.10}. Researchers surveyed two sets of students twice-weekly to see when they developed flu-like symptoms. One set was a random sample of all students and the other was a sample of friends nominated by this original set. The set of friends was size-biased sample of the students at Harvard and not a uniform sample. Sometimes called ``the friendship paradox,'' this is just the observation that the average number of friends of friends is always greater than the average number of friends \cite{Feld.91}. In the study of influenza, the set of friends showed flu-like symptoms earlier than the uniform random sample. See also \cite{GHMCCF.14}.

	Of course, there are only a finite number $K_n$, say, of connected components which correspond to each of the outbreaks. To simplify the presentation we set $G_n^i$ for $i>K_n$ as the graph on a single vertex with no edges and rooted at its only vertex.
	
	The $i^\text{th}$ largest possible outbreak is then described by the process $Z_{n,i} = (Z_{n,i}(h) : h = 0,1,\dotsm)$ defined by
	\begin{equation}\label{eqn:discProfile}
		Z_{n,i}(h) = \#\{v\in G_n^i: \dist(\rho_n^i, v) = h\},
	\end{equation} where $\dist(-,-)$ is the graph distance on $G_n^i$. In terms of the graph, $Z_{n,i}$ is the {height profile} of the component $G_{n}^i$.
	Our first result is the joint convergence of the processes $Z_{n,i}$ to a time-change of some excursion processes $\tilde{e}_i = (\tilde{e}_i(t);t\ge 0)$. The processes $\tilde{e}_i$, for $i\ge 1$, are the excursions above past minima of a certain stochastic process $\tilde{X}$ obtained by an exponential tilting of a spectrally positive $\alpha$-stable processes. See Section \ref{sec:levy} for more information on these processes.
	\begin{thm}\label{thm:Epidemic3}
		Fix some $\alpha\in (1,2)$, and some distribution $\nu$ satisfying \eqref{eqn:cdelta1}. In the product Skorohod topology on $\D(\R_+)^\infty$, the following convergence holds
		\begin{equation*}
			\left( \left(n^{-\frac{1}{\alpha+1}} Z_{n,i} (\fl{n^{\frac{\alpha-1}{\alpha+1}}t});t\ge 0 \right);i\ge 1\right) \weakarrow ((Z_i(t);t\ge 0);i\ge 1),
		\end{equation*} where $Z_i$ is the unique c\`adl\`ag solution to
		\begin{equation*}
			Z_i(t) = \tilde{e}_i\circ C_i(t) ,\qquad C_i(t) = \int_0^t Z_i(s)\,ds,\qquad \inf\{t>0: C_i(t)>0\} = 0,
		\end{equation*} where $\tilde{e}_i =(\tilde{e}_i(t): t\ge 0)$ are defined in equation \eqref{eqn:defEi} and depend on the value $\alpha$, $c$ and $\delta$ in \eqref{eqn:cdelta1}.
	\end{thm}

	\subsection{A single macroscopic outbreak and the $\alpha$-stable graph}
	
	More has been said about the graph components $G_{n}^i$ in the literature. Joseph \cite{Joseph.14} has argued that the size of the component $G_{n}^i$ scaled down by $n^{\frac{\alpha}{\alpha+1}}$ converges to a random variable $\zeta_i$ for each $i\ge 1$, which in fact can be seen to be $\zeta_i =\inf \{t>0: \tilde{e}_i(t) = 0\}$. Conchon-Kerjan and Goldschmidt \cite{CKG.20} generalize Joseph's results and show that the graph $G_n^i$ itself has a scaling limit which is a random rooted compact measured metric space $\M_i = (\M_i,\ddd_i,\rho_i,\mu_i)$. Here $\ddd_i$ is a metric on $\M_i$, $\rho_i\in\M_i$ is a specified element and $\mu_i$ is a finite Borel measure on $\M_i$.

	This means not only does height profile (the number of infected people on day $h$) converge Theorem \ref{thm:Epidemic3}, but there is some limiting continuum structure of the components $G_n^i$ which is represented by these continuum spaces $(\M_i;i\ge 1)$. The standard construction of these continuum limits, first obtained in the critical \ER case in \cite{ABBG.10,ABBG.12}, are constructions in a depth-first manner: the spaces are obtained by gluing together $k$ pairs of points on a continuum random tree with a depth-first selection of the pairs. This gluing procedure changes the distance from the origin and therefore it is non-trivial to argue convergence of the height profiles similar to Theorem \ref{thm:Epidemic3} from the results of \cite{CKG.20}.
	
	The processes $\tilde{e}_i = (\tilde{e}_i(t); t\ge 0)$ in Theorem \ref{thm:Epidemic3} and the graph $\M_i$ are of a random length and mass, respectively. That is 
	\begin{equation*}
		\tilde{e}_i(t) > 0 \qquad \text{if and only if} \quad t\in (0,\zeta_i)
	\end{equation*} for some random $\zeta_i$ and, moreover, $\zeta_i= \mu_i(\M_i)$. This does complicate the analysis somewhat; however, conditionally given the values $(\zeta_i;i\ge 1)$ the excursions $\tilde{e}_i$ (resp. the spaces $\M_i$) are independent and are described by a scaling of an excursion (resp. metric measure space) of unit length (resp. unit mass) \cite{CKG.20}. Therefore in order to understand the scaling limit $Z_1 = (Z_1(t); t\ge 0)$ of a single macroscopic outbreak $Z_{n,1}$, we can study the structure of a the process $\tilde{e}_1$ conditioned on $\zeta_1 = 1$.
	
	To do this we let $\e = (\e(t);t\in[0,1])$ denote a standard excursion \cite{Chaumont.97} of a spectrally positive $\alpha$-stable L\'{e}vy process $X = (X(t);t\ge 0)$. To simplify our proofs, we will work with situation where the Laplace transform of $X$ satisfies \begin{equation}\label{eqn:Adef}
		\E\left[\exp(-\lambda X(t) ) \right] = \exp\left(A \lambda^\alpha t\right),\qquad \forall \lambda, t\ge 0,\quad \text{where}\quad
		A = \frac{c \Gamma(2-\alpha)}{\delta \alpha (\alpha-1)},
	\end{equation} for $c,\delta$ defined in \eqref{eqn:cdelta1}. We remark that this excursion depends on the value $A$; however, the results also hold for any value of $A$ by using scaling properties of L\'{e}vy processes and their associated height processes. 
	
	We also recall from above that $\M_i$ are obtained by gluing together a finite collection of pairs of points in a continuum tree. This is the surplus of the continuum random graph $\M_i$. We will let $\GG^{(\alpha,k)} = (\GG^{(\alpha,k)}, d,\rho,\mu)$ denote the graph $\M_1$ conditioned on $\mu_1(\M_1) = 1$ and having surplus $k$. A precise construction of this object will be delayed until Section \ref{sec:realGraphs}, but it suffices to say that it will be constructed from an excursion $\e^{(k)} = (\e^{(k)}: t\in[0,1])$ defined by the polynomial tilting
	\begin{equation}\label{eqn:ktilt}
		\E\left[ f(\e^{(k)};t\in[0,1]) \right] \propto \E\left[ \left(\int_0^1 \e(t)\,dt\right)^{k} f(\e;t\in[0,1])\right].
	\end{equation} 
	
	The continuum object $\GG^{(\alpha,k)}$, in our case, represents the limiting structure of the history of the disease spread. With this, we can ask several questions in the hope that this will shed light on the structure of $M_n(\nu)$. What is the structure of the disease outbreak, i.e. what is the height profile of the graph $\GG^{(\alpha,k)}$? When does a uniformly chosen person get infected, or when do a finite number of uniformly chosen individuals get infected? When does the outbreak die out, that is, when is the last person infected? In terms of the graph $\GG^{(\alpha,k)}$ this is asking what is the radius of the graph $\GG^{(\alpha,k)}$. What's the most number of people infected at any one time or, in terms of the continuum graph, what is the distribution of the width of $\GG^{(\alpha,k)}$? There are many more different questions we can ask, but we have answers to these questions. 
	
	We start by answer the first question: what is the height profile of the graph $\GG^{(\alpha,k)}$, from which the others will follow by analysis of an integral equation. 
	
	\begin{thm}\label{thm:ctsAlpha}
		Fix $k\ge 0$, and $\alpha\in(1,2)$. Let $\GG^{(\alpha,k)}$ be the $\alpha$-stable continuum random graph, constructed from a spectrally positive L\'{e}vy process with Laplace exponent \eqref{eqn:Adef} and rooted at a point $\rho\in \GG^{(\alpha,k)}$.
		\begin{enumerate}
			\item Let $B(x,t)$ is the closed ball of radius $t$ centered at $x$. The process $\ccc = (\ccc(t); t\ge 0)$ defined by $\ccc(t) = \mu(B(\rho,t))$ is absolutely continuous and
			\begin{equation*}
				\ccc(t) = \int_0^t \z(s)\,ds
			\end{equation*} for a c\`adl\`ag process $\z = (\z(t);t\ge 0)$.
			
			\item The process $(\z(t);t\ge 0) \overset{d}{=} (z(t);t\ge 0)$ where $z$ is the unique c\`adl\`ag solution to
			\begin{equation*}
				z(t) = \e^{(k)}\left( \int_0^t z(s)\,ds\right) ,\qquad \inf\left\{t>0: \int_0^t z(s)\,ds >0\right\} = 0.
			\end{equation*}
		\end{enumerate}
	\end{thm}

	We can now answer all of the other questions once we know the height profile.
	\begin{cor}\label{cor:1}
		\begin{enumerate}
			\item The radius of the graph $\GG^{(\alpha,k)}$ is given by
			\begin{equation*}
				\sup_{v\in \GG^{(\alpha,k)}} d(\rho,v) \overset{d}{=} \int_0^1 \frac{1}{\e^{(k)}(s)}\,ds.
			\end{equation*}
			
			\item The width of the graph $\GG^{(\alpha,k)}$ is given by
			\begin{equation*}
				\sup_{t\ge 0} \z(t) \overset{d}{=} \sup_{t\in[0,1]} \e^{(k)}(t).
			\end{equation*}
			\item Let $V\in \GG^{(\alpha,k)}$ be distributed according to the mass measure $\mu$, and let $U$ denote a uniform random variable on $(0,1)$. Then
			\begin{equation*}
				d(\rho,V) \overset{d}{=} \int_0^U \frac{1}{\e^{(k)}(s)}\,ds.
			\end{equation*}
			
			\item More generally, for any $n\ge 1$, let $V_1,\dotsm, V_n$ denote random points distributed according to $\mu$ on $\GG^{(\alpha,k)}$. Let $R_{(1)}\le R_{(2)}\le\dotsm R_{(n)}$ denote the order statistics of $d(\rho,V_1),\dotsm, d(\rho,V_n)$. Let $U_{(1)}\le U_{(2)}\le \dotsm \le U_{(n)}$ denote the order statistics for an i.i.d. sample of $n$ uniform random variables. Then
			\begin{equation*}
				\left(R_{(1)},\dotsm, R_{(n)}\right) \overset{d}{=} \left(\int_0^{U_{(1)}}\frac{1}{\e^{(k)}(s)}\,ds,\dotsm, \int_0^{U_{(n)}} \frac{1}{\e^{(k)}(s)}\,ds \right).
			\end{equation*}
		\end{enumerate}
	\end{cor}
	
	After discussing the integral equation involved in the statement of part 2 of Theorem \ref{thm:ctsAlpha}, which is called the Lamperti transform in the literature, we will show how Corollary \ref{cor:1} follows from Theorem \ref{thm:ctsAlpha} in Section \ref{sec:lamperti}.

	\subsection{Relation to other works and proof structure}
	
	Epidemics on random graphs are important for many areas in the applied sciences, see \cite{BBPV.05,Volz.08,VM.07,PSCMV.15} and references therein for a non-exhaustive collection of such works. One difficulty in describing the limiting behavior comes from analyzing the influence that the specific degree distribution has on the local structure of the graph.
	
	One approach to overcoming this issue is by using a \textit{mean-field} approach \cite{BBPV.05}. A typical approximation is in continuous time where each infected vertex $v$ is infected for an exponential time, and infects its neighbors at independent exponential rates. On homogeneous networks the behavior of $Z(t)$, the number of people infected at time $t\in\R_+$, is modeled by the ordinary differential equation \begin{equation*}
		\frac{dz(t)}{dt} = \lambda \mu  z(t) (1-z(t)) \qquad \text{where}\quad z(t) = \frac{1}{n}Z(t).
	\end{equation*} A more careful analysis can be done on heterogeneous networks, where one can track the proportion of vertices with degree $k$ infected at a certain time. A remarkable thing is that this approach, while losing a lot of information about the specific local structure, it can be used to find heuristics on the proper scaling of the graphs or epidemic, see \cite{PSCMV.15}.
	
	A more detailed approach to studying heterogeneous networks was taken by Volz in \cite{Volz.08}, and rigorously proved in \cite{DDMT.12} under a fifth moment condition. In these works the population of size $n$ is broken into 3 compartments - the susceptible, the infected and the recovered -  and individuals in one compartment are moved to another compartment (i.e. an infected individual recovers or a susceptible individual is infected) at certain exponential rates. The global changes in the proportional size of the outbreak is described, to first order, by just the size of the respective compartments and the degree distribution. Here the limiting structure is described by a system of deterministic ordinary differential equations, which depend on the degree distribution. Deterministic limiting equations, perhaps after some random time $T_0$, were also obtained in \cite{JMW.14} under a second moment condition.
	
	Our approach is different and takes it idea from studies of height profiles of random trees and branching processes. Particularly, we focus on the approach implicit in \cite{AMP.04}, and later studied in \cite{CPGUB.13,CPGUB.17,AUB.20} and we use the so-called Lamperti transform. This transform was originally used for a path-by-path bijection between continuous state branching processes and a certain class of L\'{e}vy processes. This transform was originally stated by Lamperti in \cite{Lamperti.67}, but was proved later by Silverstein \cite{Silverstein.67}. See also \cite{CLU.09}.
	
	We can describe the discretized version in our situation as follows. Instead of looking at the total number of people infected on day $h$, we look at the number of individuals that person $v_j$ infects, when $v_j$ is the $j^\text{th}$ individual who contracts the disease. This corresponds to a breadth-first ordering of the underlying connected component of the graph. Call this number of newly infected individuals $\chi_j$, and let $X$ be the breadth-first walk
	\begin{equation*}
		X(k) = \sum_{j=1}^k (\chi_j-1).
	\end{equation*}
	It was this walk on the \ER random graph that Aldous used in \cite{Aldous.97} to describe the scaling limits of the component sizes of $G(n,p)$ in the critical window, and an analogous walk was used by Joseph \cite{Joseph.14} for the configuration model.
	
	An interesting property of this walk is that the number of people infected on day $h$ solves the difference equation
	\begin{equation*} 
		Z(h) = Z(0) + X\left( C(h-1)\right),\qquad C(h) = \sum_{j= 0}^h Z(j).
	\end{equation*} As far as the author is aware, the first instance of this identity can be found in \cite{AMP.04} with a slightly more complicated formulation. See the Introduction of \cite{CPGUB.13} for a proof of this equality.
	The authors of \cite{CPGUB.13,CPGUB.17} studied the scaled convergence of solutions of the above equation (with the addition of an immigration term) to its continuum analog 
	\begin{equation*}
		Z(t) =X\left( \int_0^t Z(s)\,ds\right).
	\end{equation*}
	Unfortunately, there is not a unique solution to this integral equation when $X(0) = 0$ and so proving weak convergence is quite difficult. For certain models of random trees one can prove a weak convergence result \cite{AMP.04,Kersting.11,AUB.20}, and also works for the \ER random graph when $X(0) > 0$ \cite{Clancy.20}.
	
	We overcome the uniqueness problem by arguing that the rescaled processes $Z_{n,i}(h)$ are tight, and we further show that each subsequential weak limit must be of a particular form. This approach to overcoming the uniqueness problem was used in \cite{AUB.20} to study trees with a certain degree distribution, as opposed to graphs with a given degree distribution in the present situation. While, at first, these two discrete models may seem related, the proofs are quite different. In \cite{AUB.20}, the authors use a combinatorial transformation of the tree to show that that subsequential limits must be of a particular form. We, instead, show that this follows automatically once we know the the underlying graph converges to a measured metric space. In turn, in Section \ref{sec:disc} we discuss how our abstract convergence results described in Section \ref{sec:weakGen} can be applied to the rank-1 inhomogeneous model \cite{BDW.18,BDW.20,AL.98,BvdHRS.18}.

	\section{General Weak Convergence Results}  \label{sec:weakGen}
	
	\subsection{General Weak Convergence Approach}
	
	Let us now discuss the general set up for our weak convergence arguments.
	In the introduction we discussed the epidemic, which can be realized as the height profile of a connected component of a random graph. Explicitly those graphs were viewed as a metric space, but we implicitly equipped them with the counting measure. We phrase our results in terms of more general measures on random graphs, which will likely be useful in inhomogeneous models in \cite{AL.98,BDW.18,BDW.20}. The epidemiological interpretation of considering non-uniform measures is not immediately clear; however, we could think of the unequal mass of vertices as measuring the size of a clique in a community which was reduced to single vertex.
	
	A major assumption of these results is the convergence of graphs as measured metric spaces. We delay a more detailed discussion of this topic until Section \ref{sec:ghp}. For now it suffices to say that we can equip the space $\XX$ of (equivalent classes) of pointed measured metric spaces with additional boundedness assumptions with a metric which turns $\XX$ into a Polish space. This metric is called the \ghp metric, and we will denote it by $d_\GHP$.
	
	We will denote a generic element of $\XX$ as $\M = (\M,\rho,\ddd,\mu)$ where $(\M,\ddd)$ is a metric space such that bounded sets have compact closure, $\rho\in \M$ is a specified point and $\mu$ is a Borel measure on $\M$ such that bounded sets have finite mass. For each $\alpha>0, \beta\ge 0$ define the scaling operation by
	\begin{equation*}
		\scale(\alpha,\beta)\M := (\M,\rho,\alpha\cdot \ddd,\beta\cdot \mu).
	\end{equation*}
	
	Now let $G$ denote a connected graph on, say, $n$ vertices, with $\rho\in G$ a specified vertex. We view $G$ as a measured metric space with graph distance and $\m$ a finite measure such that each vertex has strictly positive mass. As we did implicitly before, we explore the graph in a breadth-first manner. The precise way in which this is done can vary depending on the graph model, but we assume that the vertices are labeled by $v_1,\dotsm, v_n$ such that if $i< j$ then $d(\rho,v_i)\le d(\rho,v_j)$. This trivially implies that $\rho = v_1$. This labeling can be viewed as an indexing of each individual who gets infected, so that if person $B$ got infected after person $A$, then person $A$ has a smaller index than person $B$.
	
	We now discuss an underlying tree structure and breadth-first walk for the graph, which draws inspiration from the breadth-first tree and walk in \cite{AL.98}. The tree is constructed by looking at which vertices $v_i$ infects in the graph $G$. More formally, we will say that vertex $v_j$ is the child of $v_i$ if $\{v_i,v_j\}$ is an edge in $G$, but $\{v_l,v_j\}$ is \textit{not} an edge for all $l<i$. This implies $i<j$. In most models with a breadth-first exploration, $v_j$ will be a child of $v_i$ if vertex $v_j$ is discovered while exploring the vertex attached to $v_i$.
	
	We also suppose that there is some breadth-first walk $X^\BF_G$ as well:
	\begin{equation}\label{eqn:BFWass1}
		X^\BF_G(\tau(k)) = X_G^{\BF}(\tau(k-1)) - \m(v_k) + \sum_{u \text{ child of }v_k} \m(u),\qquad \tau(k) = \sum_{j\le k} \m(v_j),
	\end{equation} and with $X_G^\BF(0) = 0$. How the process $X^\BF_G$ behaves on the intervals $(\tau(k-1),\tau(k))$ will play no important role in this paper. The breadth-first walk used by Aldous and Limic in their classification of the multiplicative coalescence \cite{AL.98} satisfies equation \eqref{eqn:BFWass1}. Later an analogous walk satisfying \eqref{eqn:BFWass1} was used in \cite{AMP.04} to describe the inhomogeneous continuum random tree and extend Jeulin's identity \cite{JY.85}. When $\m$ is a uniform measure on $G$, this walk will be the breadth-first {\L}ukasiewicz path \cite{LeGall.05}.

	Importantly for us, the walk $X^\BF_G$ encodes the masses and tree structure of $v_1,\dotsm, v_n$. However there is no clean functional amenable to scaling limits which allows us to reconstruct the genealogical structure from this breadth-first walk.
	
	We now define the height profile of $G$ by
	\begin{equation*}
		Z_G(h) = \m\left\{v\in G: d(\rho,v) = h\right\}.
	\end{equation*} It will be useful to define its cumulative sum as well:
	\begin{equation*}
		C_G(h) = \sum_{j=0}^h Z_{G}(j) = \m\left\{ v\in G: d(\rho,v)\le h\right\}.
	\end{equation*}
	As observed in \cite[equations (13-14)]{AMP.04}, $Z_G(h)$ solves the following difference equation:
	\begin{equation}\label{eqn:discLamperti2}
		Z_G(h+1) = Z_G(0) + X_G^\BF \circ C_G(h).
	\end{equation}

	To describe what happens in the $n\to\infty$ limit, let $(G_n;n\ge 1)$ be a sequence of connected random graphs on a finite number of vertices, viewed as a measured metric spaces where $G_n$ is equipped with the measure $\m_n$. We write $X_n^\BF$ for the breadth-first walk $X_{G_n}^\BF$.
	We prove the following in Section \ref{sec:proofOfMainThm}.
	\begin{thm}\label{thm:conv1}
		Suppose that there exists a sequence $\alpha_n\to\infty$, and that $\gamma_n:=\m_n(G_n)\to\infty$ a.s. In addition assume:
		\begin{enumerate}
			\item In the Skorohod space $\D([0,1],\R)$, the following weak convergence holds
			\begin{equation*}
				\left(\frac{\alpha_n}{\gamma_n} X_n^\BF (\gamma_nt) ;t\in[0,1] \right) \weakarrow \left(X(t);t\in[0,1] \right),
			\end{equation*} where $X$ is a process such that almost surely $X(0) = X(1) = 0$, $X(t)>0$ for all $t\in(0,1)$ and $X(t)-X(t-)\ge 0$ for all $t$;
			\item There exists a random pointed measured metric space $\M = (\M,\rho,\ddd,\mu)$ which is locally compact and has a boundedly finite measure such that
			\begin{equation*}
				\scale(\alpha_n^{-1},\gamma_n^{-1})G_n\weakarrow \M,
			\end{equation*} weakly in the \ghp topology.
			\item For each $\eps>0$, $\mu(B(\rho,\eps)\setminus\{\rho\})> 0$ for all $\eps>0$.
			\item $\frac{\alpha_n}{\gamma_n}\sup_{v\in G_n} \m_n(v)\to 0$ as $n\to\infty$ in probability.
		\end{enumerate}
		
		Then
		\begin{enumerate}
			\item There is joint convergence in $\D(\R_+,\R)^2$:
			\begin{equation*}
				\left(\left(\frac{\alpha_n}{\gamma_n} Z_n(\fl{\alpha_n t});t\ge 0\right),\left( \frac{1}{\gamma_n} C_n(\fl{\alpha_nt});t\ge 0 \right)\right) \weakarrow \left((Z(t);t\ge 0),(C(t);t\ge 0)\right)
			\end{equation*} where $Z$ and $C$ are the unique c\`adl\`ag solution to
			\begin{equation}\label{eqn:limit1}
				Z(t) = X\circ C(t) ,\qquad C(t) = \int_0^t Z(s)\,ds,\qquad \inf\{t: C(t)>0\} = 0;
			\end{equation}
			\item The measure $\mu$ on $\M$ satisfies
			\begin{equation*}
				\left(\mu(B(\rho,t));t\ge 0 \right)\overset{d}{=} \left(C(t);t\ge 0\right)
			\end{equation*}
			
		\end{enumerate}
	\end{thm}
	
	Let us make some important remarks on the assumptions in Theorem \ref{thm:conv1}. Assumption (1) is the convergence of the breadth-first walk, which is required in order to have a description of the limiting process $Z$ as described above, barring some stochastic analysis tools that can be used in particular cases \cite{Pitman.99}. Assumptions (2) and (3) are how we overcome any possible uniqueness problems that were identified in \cite{AUB.20} (see Proposition \ref{prop:AUB} below). Particularly, assumption (3) allows for the classification of the limit $C$ satisfying $\inf\{t:C(t)>0\} = 0$. Lastly, assumption (4) is so that the term $Z_n(0) \weakarrow 0$ as $n\to\infty$. Without this assumption we are left to deal with a simpler situation to which we can use the known weak convergence results in \cite{CPGUB.13}.

	As the reader may guess, this formulation will not be helpful for the proof of Theorem \ref{thm:Epidemic3} nor in the study of any of the macroscopic outbreaks for random graphs. Instead, the above theorem works only with a single macroscopic component. In order to prove Theorem \ref{thm:Epidemic3} we must develop a joint convergence result where each of the macroscopic components of a graph converge to some limiting graphs structure. This is something that appears quite often in the literature on continuum random graphs, dating back to the celebrated result of Addario-Berry, Broutin and Goldschmidt \cite{ABBG.12}. We now suppose that we have a sequence of graphs $(G_n;n\ge 1)$ on a finite number of vertices with a measure $\m_n$. For each $n$ we denote the connected components of $G_n$ as $(G_n^i;i\ge 1)$, ordered so that
	\begin{equation*}
		\m_n\left( G_n^1\right)\ge \m_n\left(G_n^2\right)\ge\dotsm .
	\end{equation*} Again, for convenience we will say that $G_n^i$ is a graph on a single vertex where the vertex has mass $0$ for all $i>K_n$. We view each of the components as a measured metric space with graph distance, and we select a vertex $\rho_n^i$ from each component to start the breadth-first walks. Here we write $X_{n,i}^\BF$ for the breadth-first walk on $G_n^i$ which, by assumption, satisfies equation \eqref{eqn:BFWass1} with the obvious notation changes. Additionally we extend it by constancy to be a function on all of $\R_+$:
	\begin{equation*}
		X_{n,i}^\BF(t) = X_{n,i}^\BF(\m_n(G_n^i)) \qquad \forall t\ge \m_n(G_n^i).
	\end{equation*}
	Let $Z_{n,i} = (Z_{n,i}(h))$ be the height profile of the $i^\text{th}$ component $G_n^i$. They solve an equation analogous to \eqref{eqn:discLamperti2} with the obvious notation change.
	
	We prove the following
	\begin{thm}\label{thm:conv2}
		Suppose there exists two sequence $\alpha_n\to\infty$ and $\gamma_n\to\infty$ such that
		\begin{enumerate}
			\item In the product Skorohod space $\D^\infty$ the following weak convergence holds:
			\begin{equation*}
				\left(\left(\frac{\alpha_n}{\gamma_n}X_{n,i}^\BF(\gamma_nt);t\ge 0\right);i\ge 1\right)\weakarrow \left( \left(X_i(t);t\ge 0 \right);i\ge 1\right)
			\end{equation*} where almost surely, $X_i$ does not possess negative jumps and there exists a $\zeta_i>0$ such that $X(t)>0$ if and only if $t\in(0,\zeta_i)$.
			
			\item There exists a sequence of pointed measured metric spaces $\M_i = (\M_i,\rho_i,d_i,\mu_i)$ which is locally compact and has a boudedly finite measure such that
			\begin{equation*}
				\left(\scale(\alpha_n^{-1},\gamma_n^{-1})G_n^i;i\ge 1\right) \weakarrow (\M_i;i\ge 1)
			\end{equation*} weakly in the product \ghp topology.
			\item Suppose that $\mu_i(B(\rho_i,\eps)\setminus \{\rho_i\})>0$ for all $\eps>0$.
			\item $\frac{\alpha_n}{\gamma_n}\sup_{v\in G_n} \m_n(v)\to 0$ as $n\to\infty$ in probability.
		\end{enumerate}
		
		Then
		\begin{enumerate}
			\item In the product Skorohod topology \begin{equation*}
				\left(\left(\frac{\alpha_n}{\gamma_n} Z_{n,i}(\fl{\alpha_nt}), \frac{1}{\gamma_n}C_{n,i}(\fl{\alpha_n t}); t\ge 0\right) ;i \ge 1\right)\weakarrow \left(\left( Z_i(t),C_i(t);t\ge 0\right);i\ge 1\right),
			\end{equation*} where $(Z_i,C_i)$ is the unique c\`adl\`ag solution to
			\begin{equation*}
				Z_i(t) = X_i\circ C_i(t),\qquad C_i(t) = \int_0^t Z_i(s)\,ds,\qquad \inf\{t: C_i(t)>0\} =0.
			\end{equation*}
			\item For each $i\ge 1$,
			\begin{equation*}
				\left(\mu_i(B(\rho_i,t));t\ge 0\right) \overset{d}{=}\left( C_i(t);t\ge 0\right).
			\end{equation*}
		\end{enumerate}
	\end{thm}

	\subsection{Compactness Corollaries}
	
	Let us begin with the first corollary, which follows from Theorem \ref{thm:conv1} and a result in \cite{AUB.20} recalled in Proposition \ref{prop:AUB} below.
	\begin{cor}\label{cor:6}
		If the hypotheses of Theorem \ref{thm:conv1} are met, then
		\begin{equation*}
			\int_{0+} \frac{1}{X(s)}\,ds <\infty \qquad a.s.
		\end{equation*}
	\end{cor}
	
	The above corollary avoids a hypothesis in Theorem 1 in \cite{AUB.20}, but this comes at the expense of assuming convergence in the \ghp topology of an underlying metric space, which is a difficult hypothesis to verify. The inverse of Corollary \ref{cor:6} is interesting, because it gives a necessary condition for convergence in the \ghp topology.
	
	For certain models of random trees and random graphs, determining compactness of the candidates for limiting metric space is difficult. This has been a particular problem for the inhomogeneous continuum random trees introduced by Aldous, Camarri and Pitman \cite{AP.00,CP.00}. These trees are characterized by a parameter $\theta = (\theta_0,\theta_1,\dotsm)$ and in \cite{AMP.04}, the authors showed that boundedness of the continuum random tree is equivalent to the almost sure finiteness of an integral $\int_{0}^{1} \frac{1}{X(s)}\,ds$. A question was posed in \cite{AMP.04} to develop useful criteria for compactness of the ICRT and determine if boundedness implied compactness. This problem was open for 16 years, but appears to be solved very recently in \cite{BlancRenaudie.20}.
	
	It is in this vein that we state the next corollary. It is a more abstract version of part (1) of Corollary \ref{cor:1} above, and follows from part (2) in Theorem \ref{thm:conv1} and Proposition \ref{prop:AUB}.
	
	\begin{cor}
		Let $\M$ and $X$ be as in Theorem \ref{thm:conv1}, assume the hypotheses of Theorem \ref{thm:conv1} are met, and let $\operatorname{spt}(\mu)\subset\M$ denote the topological support of the measure $\mu$. Then
		\begin{equation*}
			\sup_{v\in \operatorname{spt}(\mu)} d(\rho,v) \overset{d}{=} \int_0^1\frac{1}{X(s)}\,ds.
		\end{equation*}
	\end{cor}

	\section{Preliminaries}\label{sec:prelims}

	\subsection{L\'{e}vy processes, height processes, excursions}\label{sec:levy}
	
	In this section we recall the construction of $\Psi$-height processes and their excursions. For more in depth discussion on the height processes and their excursions see the works of Le Gall, Le Jan and Duquense in \cite{DL.02,LL.98a,LL.98b}. For information about spectrally positive L\'{e}vy processes, see Bertoin's monograph \cite{Bertoin.96}.
	
	Let $X = (X(t);t\ge 0)$ denote a spectrally postive, i.e. no negative jumps, L\'{e}vy process, and let $-\Psi$ denote its Laplace transform:
	\begin{equation*}
		\E \left[\exp\left\{-\lambda X(t)\right\} \right] = \exp\left(t\Psi(\lambda)\right).
	\end{equation*} In order to discuss $\Psi$-height processes, we restrict our attention in this situation to have $\Psi$ is of the form
	\begin{equation*}
		\Psi(\lambda) = \alpha\lambda +\beta \lambda^2 + \int_{(0,\infty)} (e^{-\lambda r} - 1+ \lambda r1_{[r<1]})\,\pi(dr),
	\end{equation*} where $\alpha \ge 0$, $\beta \ge 0$, $(r\wedge r^2)\,\pi(dr)$ is a finite measure along with
	\begin{equation*}
		\beta>0 \qquad\text{or}\qquad \int_{(0,1)} r\pi(dr) = \infty.
	\end{equation*} The last assumption occurs if and only if the paths of $X$ have infinite variation almost surely.
	
	The $\Psi$-height process $H = (H(t);t\ge 0)$ is a way to give a measure (in a local time sense) to the set
	\begin{equation}\label{eqn:heightProcess}
		\{s\in[0,t]: X(s) = \inf_{s\le r\le t} X(r)\}.
	\end{equation}
	Slightly more formally, under the the additional assumption that
	\begin{equation*}
		\int_{1}^\infty \frac{1}{\Psi(\lambda)}\,d\lambda = \infty,
	\end{equation*} there exists a continuous process $H = (H(t);t\ge 0)$ such that for all $t\ge 0$ then
	\begin{equation} \label{eqn:limitDefH}
		H(t) = \lim_{\eps\to 0} \frac{1}{\eps}\int_0^t 1_{[X(s) - \inf_{r\in [s,t]} X(r)\le \eps]} \,ds,
	\end{equation} where the limit is in probability. See \cite{LL.98b} and \cite[Section 1.2]{DL.02} for more details. In the case where $\beta>0$, the process $H$ can be seen \cite[equation (1.7)]{DL.02} to satisfy
	\begin{equation*}
		H(t) = \frac{1}{\beta} \Leb \left\{\inf_{s\le r\le t} X(r): s\in[0,t]\right\}.
	\end{equation*} In particular, when $X$ is a standard Brownian motion then
	\begin{equation*}
		H(t) = 2 \left( X(t) - \inf_{s\le t} X(s) \right)
	\end{equation*} is twice a reflected Brownian motion.
	
	One can also do this same procedure to the excursions of $X$. That is, if $I(t) = \inf_{s\le t} X(s)$ is the running infimum of $X$ then the process $-I$ acts as a (Markovian) local time at level $0$ for the reflected process $X-I$ \cite[Chapter IV]{Bertoin.96}. Moreover, by looking at $T(y) = \inf\{t\ge 0: I(t) < -y\}$, we can talk about the excursions of $X-I$ between times $T(y-)$ and $T(y)$. As well \cite[Section 1.1.2]{DL.02}, it is possible to define the height process $H$ for the excursions of $X$ above its running infimum.  The associated excursion measure will be denoted by $N$. To avoid confusion, we will write $(e,h)$ for $(X,H)$ under the excursion measure $N$.

	\subsubsection{Stable Processes and Tilting}
	
	We now restrict our attention to the stable case where
	\begin{equation} \label{eqn:psiStable}
		\Psi(\lambda) = A \lambda^\alpha, \qquad\alpha\in(1,2].
	\end{equation} The process $X$ satisfies the scaling \cite{Bertoin.96}
	\begin{equation*}
		(X(t);t\ge 0) \overset{d}{=} \left(k^{-1/\alpha} X(kt);t\ge 0 \right),\qquad \forall k >0.
	\end{equation*} Similarly, the height process $H$ satisfies the scaling
	\begin{equation*}
		(H(t);t\ge 0) \overset{d}{=} \left(k^{(1-\alpha)/\alpha} H(kt);t\ge 0 \right),\qquad \forall k >0,
	\end{equation*} which can be derived from \eqref{eqn:limitDefH}.
	\begin{remark}\label{rem:Aremark}
		By scaling the L\'evy process $X$, the constant $A$ in \eqref{eqn:psiStable} can be taken to equal 1 and this is typically done in the literature. We will not do this when proving Theorems \ref{thm:Epidemic3} or \ref{thm:ctsAlpha} in order to simplify the presentation. By using scaling properties for both $X$ and $H$, it is possible to prove the results in Theorem \ref{thm:ctsAlpha} and Corollary \ref{cor:1} continue to hold when $A = 1$.
	\end{remark}

	As originally observed by Aldous \cite{Aldous.97}, one can encode the size of components of a random graph by a certain walk which possesses a scaling limit of the form $X(t)+f(t)$ where $X$ is a L\'evy process and $f(t)$ is a deterministic drift term. Aldous first proved this \cite{Aldous.97} within the critical window of \ER random graph where $X$ is a Brownian motion and $f$ is a quadratic function. This later extended to the $\alpha$-stable case by Joseph \cite{Joseph.14} on the configuration model where $X$ is a stable L\'{e}vy process.
	
	For the $\alpha$-stable case $(\alpha\in (1,2))$, Conchon--Kerjan and Goldschmidt \cite{CKG.20} described the process in \cite{Joseph.14} via an exponential tilting of a L\'{e}vy process. That is they examine an $\alpha$-stable process $X$ and its associated height process $H$ of the form and define $\tilde{X}$ and $\tilde{H}$ by
	\begin{equation}\label{eqn:processStableTilt}
		\E\left[F(\tilde{X},\tilde{H}; [0,t]) \right] = \E\left[\exp\left(-\frac{1}{\delta}\int_0^t s\,dX(s) - A \frac{t^{\alpha+1}}{(\alpha+1)\delta^\alpha} \right) F({X},{H};[0,t]) \right]
	\end{equation} where $A$ is as in \eqref{eqn:psiStable} and $F$ is a function on the paths of upto time $t$. The $A$ in our notation is $\frac{C_\alpha}{\delta}$ in the notation of \cite{CKG.20}.
	
	The excursions of the process $\tilde{X}$ before time $t>0$ can be described via the absolute continuity relationship in \eqref{eqn:processStableTilt} and the excursions of $X$ prior to time $t$. What is very useful for us is that all the excursions of \begin{equation*} \tilde{R}(t) = \tilde{X}(t)-\tilde{I}(t)
	\end{equation*} above zero can be ordered by decreasing length \cite[Lemma 3.5]{CKG.20}. That is the lengths of the excursion intervals, $(\zeta_i;i\ge 1)$, can be indexed such that $\zeta_1\ge \zeta_2\ge \dotsm\ge 0$. Corresponding the values $\zeta_i$, there is an excursion interval $(g_i,d_i)$ of length $d_i-g_i = \zeta_i$ such that $\tilde{R}(g_i) = \tilde{R}(d_i) = 0$ and $\tilde{R}(t)>0$ for all $t\in(g_i,d_i)$. We define the excursion $\tilde{e}_i = (e_i(t);t\ge 0)$ by
	\begin{equation} \label{eqn:defEi}
		\tilde{e}_i(t) = \tilde{R}((g_i+t)\wedge d_i),\qquad t\ge 0.
	\end{equation} These are the excursion which appear in Theorem \ref{thm:Epidemic3}.
	
	We also let $\tilde{h}_i = (\tilde{h}_i(t);t\ge 0)$ be the excursion of $\tilde{H}$ which straddles $(g_i,d_i)$ defined by
	\begin{equation*}
		\tilde{h}_i(t) = \tilde{H}\left((g_i+t)\wedge d_i\right).
	\end{equation*}
	
	\subsubsection{Normalized excursions and tilting}  \label{sec:stableEx}

	We now recall Chaumont's path construction of a normalized excursion of a spectrally positive $\alpha$-stable L\'{e}vy process $X$. See \cite{Chaumont.97} or \cite[Chapter VIII]{Bertoin.96} for more details on this. 
	This allows for an simple description of the conditioning the excursion measure $N(\cdot | \zeta = x)$, for a fixed constant (deterministic) $x>0$ and $\zeta$ is the duration of the excursion. These results also hold in the Brownian case $\alpha = 2$., and we refer to Chapter XII of \cite{RY.99} for that treatment.

	Define $\widehat{g}_1$ and $\widehat{d}_1$ by
	\begin{equation*}
		\widehat{g}_1 = \sup\{s\le 1: X(s) = I(s)\},\qquad \widehat{d}_1 = \inf\{s>1 : X(s) = I(s)\},
	\end{equation*} and define
	\begin{equation}\label{eqn:alphaExcursion}
		\e(t) = \frac{1}{(\widehat{d}_1 - \widehat{g}_1)^{1/\alpha}} \left( X(\widehat{g}_1 +(\widehat{d}_1-\widehat{g}_1) t) - X(\widehat{g}_1)\right),\qquad t\in [0,1].
	\end{equation}
	The normalized excursion $\e = (\e(t);t\in[0,1])$ has duration $\zeta =1$, and its law is $N(\cdot | \zeta = 1)$. We obtain, the law $N(\cdot | \zeta = x)$ by scaling. Namely, set
	\begin{equation*}
		\e_x(t) = x^{1/\alpha} \e( x^{-1} t ),\qquad  t\in [0,x],
	\end{equation*} and then $\e_x = (\e_x(t);t\in[0,x])$ has law $N(\cdot|\zeta = x)$.
	
	This can also be done under the conditioning on the lifetime of the excursion of the height process $H$. See \cite{Duquesne.03} or \cite{Miermont.03} for more information. We denote $\hp = (\hp(t);t\in[0,1])$ as the height process under the measure $N(\cdot | \zeta = 1)$ and (by the scaling for the height process) we write
	\begin{equation*}
		\hp_x (t) = x^{(\alpha-1)/\alpha} \hp(x^{-1} t),\qquad t\in[0,x].
	\end{equation*}

	The normalized excursions of $\tilde{X}$ and $\tilde{H}$ are trickier to handle because the process $\tilde{X}$ does not have stationary increments. However, there is a relatively simple way of describing these in terms of an exponential tilting of the excursions $\e$ and $\hp$ similar to Aldous' description in \cite{Aldous.97} in the Brownian case. We define the tilted processes denoted by, $\tilde{\e}^{(\delta)}_x$ and $\tilde{\hp}^{(\delta)}_x$, by
	\begin{equation}\label{eqn:tilting}
		\E\left[F(\tilde{\e}_x^{(\delta)},\tilde{\hp}^{(\delta)}_x) \right] = \frac{\E[\exp(\frac{1}{\delta}\int_0^x \e_x(t)\,dt) F({\e}_x,{\hp}_x)]}{\E[\exp(\frac{1}{\delta}\int_0^x \e_x(t)\,dt)]}
	\end{equation} When $x = 1$ or $\delta = 1$ we omit it from notation. The excursions $\ee_x^{(\delta)}$ and $\tilde{\hp}^{(\delta)}_x$ are shown in \cite{CKG.20} to be the excursions $(\tilde{e}_i,\tilde{h}_i)$ conditioned on their duration being exactly $x$.
	\begin{remark}
		To clear up any confusion between $\ee^{(\delta)}$ defined in \eqref{eqn:tilting} and $\e^{(k)}$ defined in \eqref{eqn:ktilt}, we note that we use the tilde $\tilde{\,}$ to denote tilting associated with an exponential tilting of an excursion. We do not include a tilde when discussing the polynomial tilting in \eqref{eqn:ktilt}.
	\end{remark}
	
	\subsection{Lamperti Transform} \label{sec:lamperti}

	The Lamperti transform relates continuous state branching processes and L\'{e}vy processes via a time-change. This relationship dates back to a path-by-path relationship observed by Lamperti \cite{Lamperti.67}, although only proved later by Silverstein \cite{Silverstein.67}. More recently the authors of \cite{CPGUB.13} gave a path-by-path transformation between certain pairs of L\'{e}vy processes and continuous state branching processes with immigration. The bijective relationship was known before the path-by-path connection as well, see \cite{KW.71}.
	For more information on this transformation see \cite{CLU.09} for a description in the continuum, see \cite{CPGUB.13,CPGUB.17} for scaling limits related to continuous state branching processes and their generalizations affine processes, and see \cite{AUB.20} for a scaling limits involving a similar situation of non-uniqueness of the limiting equation.
	
	We will focus on the transform applied to excursions. Given a non-decreasing function $c:\R_+\to \R$ denote its right-hand derivative by $D_+c$, i.e.
	\begin{equation*}
		D_+c(t) = \lim_{\eps\downarrow 0} \frac{c(t+\eps)- c(t)}{\eps}.
	\end{equation*}
	We now define the Lamperti transform and the Lamperti pair.
	\begin{definition}
		Given a c\`adl\`ag function $f\in \D(\R_+,\R_+)$ let
		\begin{equation*}
			\iota (t) = \int_0^t \frac{1}{f(s)}\,ds.
		\end{equation*} Define the right-continuous inverse of $\iota$, denoted by $c^0$, by
		\begin{equation*}
			c^0(t) = \inf\{s\ge 0: \iota(s)>t\},
		\end{equation*} with the convention $\inf\emptyset = \inf\{t>0:f(t) = 0\}$. The \textit{Lamperti transform} of $f$ is the function $h^0 = f\circ c^0$ and we call the pair $(h^0,c^0)$ the \textit{Lamperti pair} associated to $f$.
	\end{definition}

	Hopefully the choice of notating the Lamperti pair by $(h^0,c^0)$ will be clear after the statement of the next proposition, which we recall from \cite{AUB.20} while introducing a trivial scaling argument and fixing a typo:
	\begin{prop}(\cite[Proposition 2]{AUB.20})  \label{prop:AUB}
		Let $f\in \D(\R_+,\R_+)$ with non-negative jumps. Assume that $f(t) = 0$ if and only if $t\in\{0\}\cup[\zeta,\infty)$ for some $\zeta\in(0,\infty)$. Let $(h^0,c^0)$ denote the Lamperti pair associated to $f$. Then, solutions to
		\begin{equation}\label{eqn:lampertiODE}
			c(0) = 0, \qquad D_+ c = f\circ c
		\end{equation} can be characterized as follows:
		\begin{enumerate}
			\item[1.] If $\displaystyle \int_{0+} \frac{1}{f(s)}\,ds = \infty$ then $h = c = 0$ is the unique solution to \eqref{eqn:lampertiODE}.
			\item[2.] If $\displaystyle \int_{0+} \frac{1}{f(s)}\,ds <\infty$ then $c^0$ is not identically zero, $D_+ c^0 = h^0$, and $c^0$ solves \eqref{eqn:lampertiODE}. Furthermore, solutions to \eqref{eqn:lampertiODE} are a one-parameter family $(c^\lambda; \lambda\in [0,\infty])$ given by
			\begin{equation*}
				c^\lambda(t) = c^0 \left((t-\lambda)_+ \right), \qquad (x)_+ := x\vee 0.
			\end{equation*}
			In addition,
			\begin{enumerate}
				\item If $\displaystyle\int^{\zeta-} \frac{1}{f(s)}\,ds = \infty$ then $c^0$ is strictly increasing with $\lim_{t\to\infty} c^0(t) = \zeta$.
				\item If $\displaystyle\int^{\zeta-} \frac{1}{f(s)}\,ds <\infty$ then $c^0$ is strictly increasing until reaching $\displaystyle\iota(1) = \int_0^{\zeta} \frac{1}{f(s)}\,ds$.
			\end{enumerate}
		\end{enumerate}
	\end{prop}

	The above proposition states that all the solutions to \eqref{eqn:lampertiODE} are determined by time-shifts of the Lamperti pair associated with $f$, or is identically zero. As we will see in the sequel, a major part of the proof of Theorem \ref{thm:conv1} is showing that every subsequential (weak) limit of the $\tilde{C}_n$ is of the form $C^0$ and not a time-shift, $C^\Lambda$, of $C^0$ for some random $\Lambda$.

	With Proposition \cite{AUB.20} recalled, we can prove Corollary \ref{cor:1} from Theorem \ref{thm:ctsAlpha}.
	\begin{proof}[Proof of Corollary \ref{cor:1}]
		
		We begin by observing that 
		\begin{equation*}
			\sup_{v\in \operatorname{spt}(\mu)} \ddd(\rho,v) \overset{d}{=}\int_0^1 \frac{1}{{\e}^{(k)}(s)}\,ds
		\end{equation*} follows from Theorem \ref{thm:ctsAlpha} by an application of conclusion (2)(b) in Proposition \ref{prop:AUB}. To replace the support of the measure $\mu$ with the graph $\GG^{(\alpha,k)}$ we observe that \begin{equation*}
			\sup_{v\in \operatorname{spt}(\mu)} \ddd(\rho,v) = \sup_{v\in \GG^{(\alpha,k)}} \ddd(\rho,v).
		\end{equation*} Indeed to prove this equality observe that the leafs of the graph $\GG^{(\alpha,k)}$ are dense in both the support of the measure $\mu$ and the graph $\GG^{(\alpha,k)}$ which follows from analogous results for the continuum random trees \cite{Duquesne.03,DL.05} and the observation that the exponential tilting in the construction of the graphs does not change this almost sure statement.
		
		Part (2) trivially follows from Theorem \ref{thm:ctsAlpha} and the observation that $\ccc$ increases from $0$ to $1$ as $t$ ranges from $0$ to $\infty$. We restrict the rest our proof to part (3), the argument of which will imply part (4) with minor modifications.

		We recall the well-known fact that if $X$ is a real random variable taking values in $(a,b)$ with cumulative distribution function $F$ which is strictly increasing on $(a,b)$, then $X \sim F^{-1}(U)$ where $U$ is a standard uniform random variable and $F^{-1}(y) = \inf\{t: F(t)>y\}$ is the right-continuous inverse. Typically this is stated with the left-continuos inverse of $F$; however, when $F$ is strictly increasing these two inverses agree on $(a,b)$.
		
		Now, conditionally given $\GG^{(\alpha,k)}$ , Theorem \ref{thm:ctsAlpha} implies that
		\begin{equation*}
			\P(\ddd(\rho,V)\le t| \GG^{(\alpha,k)}) = \mu(B(\rho,t)) = \ccc(t).
		\end{equation*} Thus,
		\begin{equation*}
			\ddd(\rho,V) \overset{d}{=} \inf\{t: \ccc(t) > U\} ,\qquad U\sim \operatorname{Unif}(0,1).
		\end{equation*}
		However, the process $\ccc$ is equal in distribution to $c(t) = \int_0^t z(s)\,ds$ where $z$ is as in part (2) of Theorem \ref{thm:ctsAlpha}. It is easy to see by examining the discussion of the Lamperti transform above, that
		\begin{equation*}
			\ccc(t) = \inf\left\{u: \int_0^u \frac{1}{\e^{(k)}(s)}\,ds>t\right\} \wedge 1.
		\end{equation*} See also the discussion preceding Proposition 2 in \cite{AUB.20} and Chapter 6 of \cite{EK.86} as well. The result now follows by taking another inverse.
		
		The proof of part (4) is a trivial generalization involving order statistics.
	\end{proof}

	\subsection{Convergence of Metric Spaces} \label{sec:ghp} In this section we discuss how to topologize the collection of pointed measured metric spaces with some additional compactness assumptions. We start with a definition:
	
	\begin{definition}
		A collection $\M = (\M,\rho,\ddd,\mu)$ is a \textit{pointed measured metric} (PMM) space if $(\M,\ddd)$ is a metric space, $\mu$ is a Borel measure on $\M$ and $\rho\in \M$ is a distinguished point. We say that $\M$ is \textit{boundedly compact} if bounded sets are pre-compact and we say that $\mu$ is a \textit{boundedly finite} measure if bounded sets have finite mass. We say that $\M$ and $\M'$ are \textit{isomorphic} if there exists a bijective isometry $f:\M\to \M'$ such that $f(\rho) = f(\rho')$ and $f_{\#}\mu = \mu'$ and $f^{-1}_{\#} \mu' = \mu$. We denote the collection of all (isomorphism classes of) boundedly compact PMM spaces equipped with boundedly finite measures by $\XX$. Let $\XX_c\subset \XX$ consist of all compact elements of $\XX$.
	\end{definition}

	We leave a more detailed accounting of the metric space structure of $\XX$ to the texts \cite{ADH.13,Lei.19}. We do recall some useful properties which will be used in the sequel.
	
	\begin{thm} (\cite{ADH.13,Lei.19})
		\begin{enumerate}
			
			\item There exists a metric $d_\GHP$ on the space $\XX$ which makes $(\XX,d_\GHP)$ a Polish space.
			
			\item There exists a metric $d_\GHP^c$ on the space $\XX_c$ which makes $(\XX_c,d_\GHP^c)$ a Polish space.
			
		\end{enumerate}
	\end{thm}
	
	We now prove the following simple lemma, which we cannot find in the existing literature. This will be used in the proof of Theorem \ref{thm:conv1}. We denote by $B(y,r)$, the closed ball of radius $r$ centered at $y$ in the appropriate metric space.
	\begin{lem}\label{lem:measrueOfBalls} Let $\M_n = (\M_n,\rho_n,\ddd_n,\mu_n)$, $\M = (\M,\rho,\ddd,\mu)$ be random elements of $\XX$ such that
		\begin{equation*}
			\M_n\weakarrow \M,\qquad\text{ as random elements of }(\XX,d_{\GHP}),
		\end{equation*} and $\mu$ on $\M$ is almost surely not the zero measure.
		Then for all but countably many $r\in(0,\infty)$:
		\begin{equation*}
			\mu_n(B(\rho_n,r))\weakarrow \mu(B(\rho,r)),\qquad\text{as real numbers}.
		\end{equation*}
		
		Both convergences above can be replaced with almost sure convergence as well.
		
	\end{lem}
	
	We prove this by first appealing to a deterministic lemma.
	\begin{lem}\label{lem:ballCont}
		Let $\M_n\to \M$ in $d_{\GHP}$ and suppose that the measure $\mu$ on $\M$ is not the zero measure. Let $r$ be a radius such that
		\begin{equation*}
			\mu\left\{x\in \M: \ddd(\rho,x) = r\right\} = 0.
		\end{equation*}
		Then
		\begin{equation*}
			\mu_n(B(\rho_n,r))\to \mu(B(\rho,r)).
		\end{equation*}
	\end{lem}
	\begin{proof}
		By Theorem 3.16 in \cite{Lei.19}, it suffices to consider the compact case where $\M_n,\M\in \XX_c$ and $\M_n\to \M$ with respect to the $d_{\GHP}^c$ metric and that $r = \sup_{x\in \M} \ddd(\rho,x)$.
		
		We recall that, for metric spaces $X$ and $Y$, a function $f:X\to Y$ is an $\eps$-isometry if $f$ is measurable and
		\begin{equation*}
			\sup\{|d(x_1,x_2) - d(f(x_1),f(x_2))| :x_1,x_2\in X\}\le \eps
		\end{equation*} and
		for all $y\in Y$ there exists some $x\in X$ such that $d(y,f(x))<\eps$.
		
		By Theorem 3.18 in \cite{Lei.19}, there exists a sequence $\eps_n\to 0$ and a sequence of functions $f^n:\M_n\to \M$ such that $f^n$ is an $\eps_n$-isometry and such that
		\begin{equation*}
			f_\#^n\mu_n \to \mu,
		\end{equation*} with respect to the weak-* topology of measures on $\M$, that is convergence of the integrals against compactly supported continuous functions. However, $1_{\M}$ is continuous and compactly supported since $\M$ is compact. So the following convergence holds in because of convergence in the weak-* topology:
		\begin{equation*}
			\begin{split}
				\mu_n(\M_n) &= \int_{\M_n} 1_{\M_n} \,d\mu_n = \int_{\M_n} 1_{\M}\circ f^n(x)\,\mu_n(dx) \\
				&= \int_\M 1_{\M} (x) (f_\#^n\mu_n)(dx)\to \int_\M 1_{\M}\,d\mu = \mu(\M).
			\end{split}
		\end{equation*}
		Therefore, there is no loss in generality in assuming that the measures $\mu_n$ and $\mu$ are probability measures, since we can just rescale the measures by their (non-zero) total mass. Since weak-* convergence of probability measures on a compact space is simply weak convergence of probability measures, the desired convergence holds by Portmanteau.
	\end{proof}

	\begin{proof}[Proof of Lemma \ref{lem:measrueOfBalls}]
		By Lemma \ref{lem:ballCont}, we have shown that the map $\Phi_r:\XX\to \R$ by
		\begin{equation*}
			\Phi_r(\M) = \mu(B(\rho,r))
		\end{equation*} is continuous at each $\M$ such that $\mu(\{x\in \M: \ddd(\rho,x) = r\}) = 0$.
		
		Now given a random element $\M$ with law $\P$, we just need to show
		\begin{equation*}
			\left\{r: \P\left[\mu(\{x\in \M: \ddd(\rho,x) = r\})> 0 \right]>0\right\} \quad\text{is countable}.
		\end{equation*} This follows from the same argument that random processes in $\D(\R_+,\R)$ cannot have a uncountably many jump-times which occur with strictly positive probability. The proof of that latter statement can be found in Section 13 of \cite{Billingsley.99}, but is omitted here.
		
	\end{proof}
	
	\subsection{Continuum random trees and continuum random graphs} \label{sec:realGraphs}

	In this section we briefly recall the definition of continuum random trees and continuum random graphs. This will not be a full description of what these metric spaces are, but will be enough to define the metric spaces we use in the sequel. For a more abstract account of these metric spaces see Section 2.2 of \cite{ABBGM.17}, for example.
	
	We briefly describe a real tree encoded by a continuous function $h$, see \cite{LeGall.05} and references therein for more information. Let $h:[0,x]\to [0,\infty)$ be a continuous function such that $h(0) = h(x) = 0$ and $h(t)>0$ for all $t\in (0,x)$. We can define a psuedo-distance $d_h$ on $[0,x]$ by
	\begin{equation*}
		d_h(s,t) = h(s)+h(t) - 2 \inf_{s\wedge t\le r \le s\vee t} h(r).
	\end{equation*} We then define an equivalence relation by $s\sim_h t$ if $d_h(s,t) = 0$. The random tree $\T_h$ is defined as the quotient space
	\begin{equation*}
		\T_h = [0,x]/\sim_h.
	\end{equation*} and let $q_h:[0,x]\to \T_h$ denote the canonical quotient map. The topological space $\T_h$ can be made into a PMM by setting the specified point as $\rho:= q_h(0) = q_h(x)$, the distance as $d(q_h(s),q_h(t)) = d_h(s,t)$ which is a well-defined metric and the measure as $\mu = (q_h)_{\#}\Leb|_{[0,x]}$. We call $\T_h$ the tree encoded by the function $h$, and call $h$ the height function (or process) of $\T_h$.
	
	The spaces $\T_h$ are tree-like in the sense that given any two elements $a,b\in \T_h$, there exists a unique isometry $f_{a,b}:[0,d(a,b)]\to \T_h$ such that $f_{a,b}(0) = a$ and $f_{a,b}(d(a,b)) = b$ and every continuous injection $f:[0,1]\to \T_h$ such that $f(0) = a$ and $f(b) = 1$ is a reparametrization of $f_{a,b}$.
	
	Let us now describe how to add \textit{shortcuts} to the tree $\T_h$ in order to form a graph-like metric space. We fix a c\`adl\`ag function $g:[0,x]\to [0,\infty)$ such that $g(0) = g(x) = 0$ and $g$ doesn't jump downwards, i.e. $g(t)-g(t-)\ge 0$ for all $t$. We also suppose that we have a finite set $Q = \{(t_j,y_j):j=1,\dotsm, s\}$ of points in $[0,x]\times \R_+$ such that $y_j\le g(t_j)$. For each of these values $t_j$, we define the value $\tilde{t}_j$ by
	\begin{equation*}
		\tilde{t}_j = \inf\{u\ge t_j: g(u) = y_j\}.
	\end{equation*} The infimum above is taken over a non-empty set because $g$ does not jump downwards. These times $t_j$ and $\tilde{t}_j$ will come to represent the points in the tree that are glued together.
	
	Let us now go back into the continuum tree $\T_h$. Define the vertices $u_j = q_h(t_j)$ and $v_j = q_h(\tilde{t}_j)$ where $q_h$ is the canonical quotient map. We define a new equivalence relation $\sim$ on $\T_h$ which depends on both $g$ and $Q$ by setting $u_j\sim v_j$ for each $j = 1,2,\dotsm,s$. We define the set
	\begin{equation*}
		\mathscr{G}(h,g,Q) = \T_h /\sim.
	\end{equation*}
	It is straightforward to turn $\mathscr{G}(h,g,Q)$ into a PMM by where the distance between $u_j$ and $v_j$ is zero.
	
	In the description of the construction of the graph $\mathscr{G}(h,g,Q)$ it is easier to consider only the case where $Q$ consists of points $(t,y)$ such that $0\le y\le g(t)$ and $t\in[0,x]$. We can equally as well consider the situation where $Q$ is a discrete set in $\R_+\times \R_+$ with finitely many elements in any compact set, and define
	\begin{equation*}
		\mathscr{G}(h,g,Q) = \mathscr{G}(h,g,g\cap Q),\quad\text{where}\quad g\cap Q = \{(t,y)\in Q: y\le g(t)\}.
	\end{equation*}

	Let us now describe the graphs $\GG^{(\alpha,k)}$ that we mentioned in the introduction. See \cite{GHS.18, CKG.20} for more information on these graphs. We first define the $\alpha$-stable graph $\GG^{(\alpha)}$ where we let the surplus be a random non-negative integer.
	The graphs $\GG^{(\alpha)}$ are the graphs $\GG(\tilde{\hp}^{(1)}_1,\ee^{(1)}_1,\PP)$ for a Poisson random measure $\PP$ on $\R_+^2$ with Lebesgue intensity. The Poisson point process $\PP$ has only a finite number of points $(t,y)$ such that $0\le y\le \ee^{(\delta)}_1(t)$, and this is the surplus of the random graph $\GG^{(\alpha)}$. The graph $\GG^{(\alpha,k)}$ is just the graph $\GG^{(\alpha)}$ conditioned on having fixed surplus $k$. This conditioning on the number of points of $\PP$ which lie under the curve $\ee^{(1)}_1$ changes the exponential tilting in \eqref{eqn:tilting} to the polynomial tilting in \eqref{eqn:ktilt}.

	For more information on random trees and graphs see \cite{Aldous.91,Aldous.91a,Aldous.93, LeGall.05} for the Brownian CRT, see \cite{Miermont.03,Duquesne.03,DL.02} for the stable-trees, see \cite{ABBG.12,BBSW.14} for the Brownian random graph and \cite{CKG.20,GHS.18} for the stable graph.

	\section{Proofs of Weak Convergence Results} \label{sec:proofOfMainThm}
	
	We now turn our attention to proving the abstract weak convergence results: Theorems \ref{thm:conv1} and \ref{thm:conv2}. To simplify the notation in the proof of Theorem \ref{thm:conv1}, we write $X_n(\cdot) = Z_n(0) + X_n^\BF(\cdot)$. By assumption (4) in Theorem \ref{thm:conv1}, $Z_n(0)\to 0$ in probability and so assumption (1) in Theorem \ref{thm:conv1} holds with $X_n$ replacing $X_n^\BF$ by Slutsky's theorem. Moreover, changing \eqref{eqn:discLamperti2} to match this notation, the process $Z_n$ solves
	\begin{equation*}
		Z_n(h+1) = X_{n}\circ C_n(h),\qquad C_n(h) = \sum_{j=0}^h Z_n(j), \qquad C_n(-1) = 0.
	\end{equation*}
	
	We define the rescalings:
	\begin{equation*}
		\begin{split}
			\tilde{Z}_n(t) &= \frac{\alpha_n}{\gamma_n} Z_n(\fl{\alpha_nt})\\
			\tilde{C}_n(t) &= \frac{1}{\gamma_n} C_n(\fl{\alpha_nt}) \\
			\tilde{X}_n(t) &= \frac{\alpha_n}{\gamma_n}X_n(\gamma_nt).
		\end{split}
	\end{equation*}
	
	We begin by proving the tightness of $\tilde{X}_n$ and $\tilde{C}_n$.
	\begin{prop}\label{prop:jointCX}
		Under the assumptions of Theorem \ref{thm:conv1}, and the above notation, the sequence $((\tilde{C}_n,\tilde{X}_n);n\ge 1)$ is tight in $\D(\R_+,\R)^2$. Moreover, any subsequential limit of $(\tilde{C}_n,\tilde{X}_n;n\ge 1)$, say $(C,X)$, must satisfy \begin{equation*}C(t) = \int_0^t X\circ C(s)\,ds.\end{equation*}
	\end{prop}
	
	\begin{proof}
		We alter the proof of Proposition 7 in \cite{AUB.20}. That proof involves a linear interpolation of $C_n$ instead, which makes their proof slightly simpler. The differences are easily overcome using compactness results in Billingsley's monograph \cite{Billingsley.99}.
		
		Because tightness of marginals implies tightness of the pair of random elements, in order to show the tightness claimed, it suffices to show that $(\tilde{C}_n;n\ge 1)$ is tight, since we assume that $\tilde{X}_n$ converges weakly and is therefore tight.
		Towards this end, observe that $\tilde{C}_n$ is uniformly bounded:
		\begin{equation}\label{eqn:uniformlyBounded}
			0\le \tilde{C}_n(t) \le \frac{1}{\gamma_n} \m_n(G_n)  =1.
		\end{equation}
		
		We now set $t>s$. We have
		\begin{align*}
			\tilde{C}_n(t) - \tilde{C}_n(s) &= \frac{1}{\gamma_n} (C_{n}(\fl{\alpha_nt}) - C_n(\fl{\alpha_ns}))\\
			&=  \frac{1}{\gamma_n} \sum_{h=\fl{\alpha_ns}+1}^{\fl{\alpha_nt}} Z_n(h)\\
			&\le \frac{1}{\gamma_n} \int_{\alpha_ns-1}^{\alpha_nt+1} Z_n(\fl{u})\,du\\
			&= \frac{1}{\gamma_n} \int_{\alpha_ns-1}^{\alpha_nt+1} X_n\circ C_n(\fl{u})\,du\\
			&= \frac{\alpha_n}{\gamma_n} \int_{s-\alpha_n^{-1}}^{t+\alpha_n^{-1}} X_n\circ C_n(\fl{\alpha_nu})\,du\\
			&\le (t-s+2\alpha_n^{-1}) \|\tilde{X}_n\|,
		\end{align*} where
		\begin{equation*}
			\|f(t)\| = \sup_{t\in[0,1]} |f(t)|.
		\end{equation*}

		Define the functions
		\begin{equation*}
			w(f; I):= \sup_{s,t\in I} |f(t)-f(s)|,\qquad f\in \D(\R_+), I\subset \R,
		\end{equation*}
		and, for $\delta>0$,
		\begin{equation*}
			w_N'(f;\delta):= \inf_{\{t_i\}} \max_{1\le i\le v} w(f;[t_{i-1},t_i)), \qquad N = 1,2,\dotsm
		\end{equation*} where the infimum is taken over all partitions $0= t_0< t_1<\dotsm < t_v= N$ such that $t_i-t_{i-1}>\delta$ for $1\le i< v$.
		
		From the above string of inequalities, for any integer $N>0$ and any $\delta>0$, we have
		\begin{equation*}
			w'_N(\tilde{C}_n;\delta)\le  2\left(\delta+\alpha_n^{-1} \right) \|\tilde{X}_n\|, \qquad \forall n\ge 1,
		\end{equation*} Moreover, for any fixed $\delta>0$, there exists an $n_0$ sufficiently large such that
		\begin{equation*}
			2(\delta+ \alpha_n^{-1}) \|\tilde{X}_n\| \le 4 \delta \|\tilde{X}_n\|,\qquad \forall n\ge n_0,
		\end{equation*} since $\alpha_n\to\infty$.
		Fix $\eps>0$ and an integer $N$. Applying Theorem 13.2 in \cite{Billingsley.99} gives
		\begin{align*}
			\lim_{\delta\downarrow 0}\limsup_{n\to\infty} \P\left(w_N'(\tilde{C}_n;\delta) \ge \eps \right) \le \lim_{\delta\downarrow 0} \limsup_{n\to\infty} \P \left(\|\tilde{X}_n\| \ge \frac{\eps}{4\delta} \right) = 0.
		\end{align*} Hence, by Theorem 16.8 in \cite{Billingsley.99}, the process $\tilde{C}_n$ is tight.
		
		The statement about the form of the subsequential weak limits follows as in the proof of Proposition 7 in \cite{AUB.20} with little alteration.
	\end{proof}

	\subsection{Proofs of Theorems \ref{thm:conv1} and \ref{thm:conv2}}
	
	We now move to describe more accurately the possible subsequential limits in Proposition \ref{prop:jointCX}.
	By Proposition \ref{prop:AUB} and Proposition \ref{prop:jointCX}, the subsequential limits $(\tilde{C}_{n_\ell},\tilde{X}_{n_\ell}) \weakarrow (C,X)$ must be of the form $C(t) = C^0((t-\Lambda)_+)$ for some (random) $\Lambda:= \inf\{t: C(t)>0\}\in[0,\infty]$ where $C^0$ is the Lamperti transform of the $X$. We desire to show that $\Lambda =0$ almost surely.

	By the Skorokhod representation theorem and by possible taking a further subsequence, we can assume that we are working on a probability space such that both
	\begin{equation*}
		\scale(\alpha_{n_\ell}^{-1},\gamma_{n_\ell}^{-1})G_{n_\ell}\to \M,\qquad \text{and}\qquad (\tilde{C}_{n_\ell},\tilde{X}_{n_\ell}) \to (C , X)
	\end{equation*} occur almost surely in their respective topologies: the first convergence is with respect to the pointed \ghp topology and the second convergence is with respect to the product topology on the $\D\times\D$. We write $\tilde{G}_{n_\ell} = (\tilde{G}_{n_\ell}, \rho_{n_\ell}, \tilde{d}, \tilde{\m}_{n_\ell})$ for $\scale(\alpha_{n_\ell}^{-1},\gamma_{n_\ell}^{-1})G_n^i$. By Lemma \ref{lem:measrueOfBalls}, we have for all but countably many $t>0$,
	\begin{equation*}
		\begin{split}
			\tilde{C}_{n_\ell}(t) &= \frac{1}{\gamma_{n_\ell}} \m_{n_\ell}\{v\in G_{n_\ell}: d(v,\rho_{n_\ell})\le \alpha_{n_\ell} t\} \\
			&= \tilde{\m}_{n_\ell}\left(\{v\in \tilde{G}_{n_\ell}: \tilde{d}(v,\rho_{n_\ell}) \le t\} \right) \longrightarrow \mu(\{x\in \M: d(\rho,x)\le t\}).
		\end{split}
	\end{equation*} Similarly, by the convergence of $\tilde{C}_{n_\ell}$ in $\D$ we have for all but countably many $t>0$
	\begin{equation*}
		\tilde{C}_{n_\ell}(t) \rightarrow C(t).
	\end{equation*} By a standard diagnolization argument there exists a sequence $t_{m}\downarrow 0$ such that
	\begin{equation*}
		C(t_{m}) = \mu \left( \bar{B}(\rho,t_{m})\right)> 0.
	\end{equation*} The inequality follows from Assumption (3) in Theorem \ref{thm:conv1}.
	
	Hence
	\begin{equation*}
		\Lambda = \inf\{t: C(t)>0\} = 0.
	\end{equation*}
	
	Therefore every subsequential weak limit for $(\tilde{C}_n,\tilde{X}_n)$ must be of the form $(C^0,X)$ where $C^0$ is part of the Lamperti pair associated with $X$. Along with looking at Proposition \ref{prop:AUB}, we have proved the following:
	\begin{prop}\label{prop:jointCXwithLimit}
		Let $(Z^0,C^0)$ be the Lamperti pair of $X$. Then, under the assumptions of Theorem \ref{thm:conv1}, the following weak convergence holds
		\begin{equation*}
			\left(\tilde{C}_n,\tilde{X}_n\right) \Longrightarrow \left(C^0,X\right),
		\end{equation*} in the product Skorokhod space $\D\times\D$.
		
		Moreover, since $C^0$ is not identically zero $X$ must satisfy
		\begin{equation*}
			\int_{0+} \frac{1}{X(s)}\,ds <\infty.
		\end{equation*}
	\end{prop}

	We now finish the proof of Theorem \ref{thm:conv1}.
	\begin{proof}[Proof of Theorem \ref{thm:conv1}]
		
		The proof of Proposition \ref{prop:jointCXwithLimit} gives the proof of conclusion (2) of Theorem \ref{thm:conv1}, and so we finish the proof of part (1).

		By Proposition \ref{prop:jointCX} and Proposition \ref{prop:jointCXwithLimit}, and the Skorohod representation theorem, we can assume that we are working on a probability space such that
		\begin{equation*}
			\left(\tilde{C}_n,\tilde{X}_n\right) \longrightarrow (C^0,X)\qquad \text{a.s.}.
		\end{equation*}
		Then, a result of Wu \cite[Theorem 1.2]{Wu.08} which extends a result of Whitt \cite{Whitt.80}, the following convergence holds in the $\D$:
		\begin{equation*}
			\tilde{Z}_n = \tilde{X}_n\circ \tilde{C}_n \longrightarrow X\circ C^0 \qquad \text{a.s.}.
		\end{equation*}
		
		Using Proposition \ref{prop:AUB} part (2), we observe that
		\begin{equation*}
			Z^0:=X\circ C^0 = D_+ C^0.
		\end{equation*} That is $(Z^0,C^0)$ is the Lamperti pair associated with $X$ and in $\D^2$
		\begin{equation*}
			(\tilde{Z}_n,\tilde{C}_n) \weakarrow (Z^0,C^0).
		\end{equation*}
	\end{proof}
	
	The proof of Theorem \ref{thm:conv1} can be easily extended to joint convergence of the of finitely many graphs $G_n^i$ of random masses $\m_n(G_n^i)$ as $n\to\infty$. Since the graph $G_n^{i}$ are ordered by decreasing mass, the excursion lengths also decrease: $\zeta_1\ge \zeta_2\ge \dotsm$. The only part that changes is \eqref{eqn:uniformlyBounded} is replaced with an analogous tightness bound on \begin{equation*}
		\max_{j \le N} \frac{1}{\gamma_n} \m_n(G_n^j) = 
		\frac{1}{\gamma_n} \m_n(G_n^1)
	\end{equation*} This will yield a proof of Theorem \ref{thm:conv2}. The details are omitted.

	\section{The Configuration Model}
	
	In this section we focus on the applications to the configuration model when one specifies a critical degree distribution $\nu$ in the domain of attraction of a stable law. We will focus on the case $\alpha\in(1,2)$, although the Brownian case $\alpha=2$ can be obtained by these methods. The results can easily be altered to cover the $\alpha = 2$ as well, by instead considering the case where $\nu$ has finite third moment at the critical point (see the definition of $\theta$ in \eqref{eqn:thetaDef} below) and omitting the cases $\nu(2)<1$ and $\nu(0)> 0$.

	We will be using the results of Joseph \cite{Joseph.14} and Conchon-Kerjan and Goldschmidt \cite{CKG.20} on scaling limits related to the configuration model. The latter reference provides a metric space scaling limit for the components of the graph at the point of criticality $\theta = 1$, where $\theta$ is defined in \eqref{eqn:thetaDef}. This allows us to utilize Theorem \ref{thm:conv2}. Similar results in the $\alpha=2$ case were obtained prior to Joseph, see Riordan's work \cite{Riordan.12} and also \cite{BS.20}.
	
	\subsection{Preliminaries: The configuration model and convergence}\label{sec:configuration}
	
	Let us describe briefly the configuration model, some of the associated walks on the graphs, and their scaling limits. For a more detailed account of the configuration model, see Chapter 7 of \cite{vanderHofstad.17}.

	The multigraph $M(\dd^n)$ is a random graph on vertex set $(i;i\in[n])$ where the vertex $i$ has degree (counted with multiplicity) $d_i$. We can construct this graph by viewing the vertices $i$ as hubs with $d_i$ \textit{half-edges} jutting out from the vertex $i$. We then pair half-edges uniformly at random to create a multigraph. Given a multigraph $G$, we have \cite[Prop 7.7]{vanderHofstad.17}
	\begin{equation*}
		\P(M(\dd^n) = G) = \frac{1}{\left(  -1+\sum_{j=1}^n d_j \right)!!}\times  \frac{ \prod_{j=1}^n d_j!}{ \prod_{j=1}^n \operatorname{loop}(i) \times \prod_{1\le i<j\le n} \operatorname{mult}(i,j)! },
	\end{equation*} where $\operatorname{loop}(i)$ is the number of self-loops at vertex $i$ and $\operatorname{mult}(i,j)$ is the number of edges between $i$ and $j$. Below we describe two different algorithms for how to construct the multigraph and describe associated walks. It is also convenient to assume that the $d_i$ half-edges connected to a vertex $i$ are ordered, so that we can talk about the ``least'' half-edge. We remark that this random construction described above is taken from a deterministic sequence of half-edges $\dd^n$, later on we will take the vertex degrees to be random.
	
	We describe two algorithms for the construction in a manner quite similar to Joseph \cite{Joseph.14}. We partition the $\sum_{j=1}^n d_j$ half-edges into three disjoint subsets: the set $\mathcal{S}$ of sleeping half-edges, the set $\mathcal{A}$ of active half-edges, and the set $\mathcal{D}$ of dead half-edges. We call the set $\mathcal{S}\cup \mathcal{A}$ the collection of alive half-edges. Initially all half-edges are sleeping.

	\subsubsection{Breadth-first construction}
	
	We construct a graph $M^\BF(\dd^n)$ (we initially include a $\BF$ to specify the construction) as follows:
	
	To initialize at step 1, we pick a sleeping half-edge uniformly at random. Label the corresponding vertex as $v_1$ and declare all of the half-edges attached to $v_1$ as active.
	
	Suppose that we have just finished step $j$. There are three possibilities: (1) $\mathcal{A} \neq \emptyset$, (2) $\mathcal{A} = \emptyset$ and $\mathcal{S}\neq \emptyset$, (3) all half-edges are dead.
	
	In case 1, we proceed as follows:
	\begin{enumerate}
		\item Let $i$ be the \textbf{smallest} integer $k$ such that there exists an active half-edge attached to $v_k$.
		\item Pick the least half-edge $l$ from all active half-edges attached to $v_i$.
		\item Kill $l$, that is, remove it from $\mathcal{A}$ and add it to $\mathcal{D}$.
		\item Choose uniformly at random from all living half-edges $r$ and pair it with $l$, that is, add an edge between the vertex $v_i$ (which is attached to $l$) and the corresponding vertex connected to $r$.
		\item If $r$ is sleeping, then we have discovered a new vertex. Label this new vertex $v_{m+1}$ where we have discovered the vertices $v_{1},\dotsm, v_m$ up to this point. Declare all the half-edges of $v_{m+1}$ are active.
		\item Kill $r$.
	\end{enumerate}
	
	In case 2, we have finished exploring a connected component of $M^\BF(\dd^n)$. We proceed by picking a sleeping half-edge uniformly at random. We then label the corresponding vertex $v_{m+1}$ if we've discovered vertices $v_1,\dotsm, v_m$ up to this point, and we declare all the half-edges connected to $v_{m+1}$ as active.
	
	In case 3, we have explored the entire graph and we are done.
	
	The above is the breadth-first construction of the multigraph $M^\BF(\dd^n)$. In the sequel, we denote the ordering of the vertices in the exploration/construction above by $v_1^\BF,\dotsm, v_n^\BF$.
	
	\begin{remark}
		While the above algorithm gives a breadth-first construction of the graph $M(\dd^n)$, observe that this can also be used to explore the graph $M(\dd^n)$. Indeed, if in step (4), we selected the half-edge $r$ which is connected to $l$ instead of sampling it uniformly, then we would have explored the graph and obtain the an equal in distribution ordering of the vertices. 
	\end{remark}
	
	\begin{figure}[h!]
		\begin{subfigure}{.45\textwidth}
			\centering
			\includegraphics[width=0.9\linewidth]{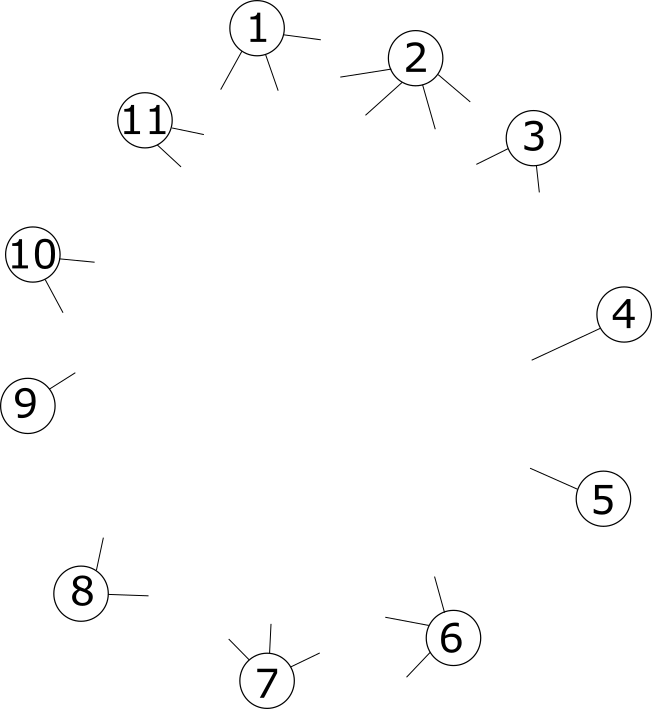}
		\end{subfigure}
		\begin{subfigure}{.45\textwidth}
			\centering
			\includegraphics[width=0.9\linewidth]{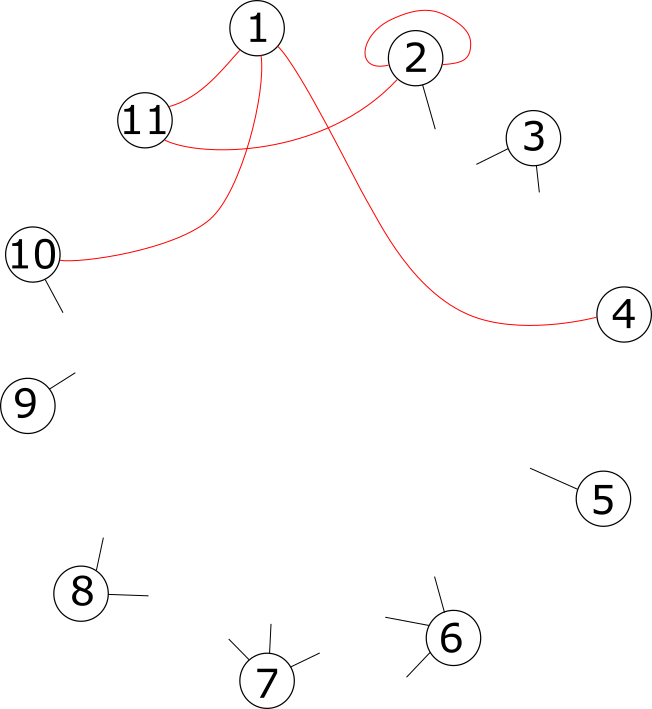}
		\end{subfigure}
		
		\caption{Left: The initial collection of 11 vertices with half-edges appearing from the center. Right: The structure of the breadth-first constructed graph after initially selected a half-edge connected to vertex $11$. The edges were added in this order $\{11,1\}, \{11,2\}, \{1,4\}, \{1,10\}, \{2,2\}$. The next half-edge to be explored is the remaining half-edge jutting out form vertex 2.}
	\end{figure}

	In case 1 above, it is possible that we match two half-edges $l$ with $r$ where $r$ is already active. We call the corresponding edge in the multigraph $M^\BF(\dd^n)$ a \textit{breadth-first (bf) backedge}.
	
	We let $F^\BF(\dd^n)$ denote the forest constructed from the multigraph $M^\BF(\dd^n)$ obtained by splitting all bf backedges into two half-edges and adding two leaves to each of these half-edges. More formally, if the multigraph $M^\BF(\dd^n)$ has a bf backedge between vertices $v_l, v_r$. Remove that edge from the multigraph and add two vertices $v_l'$ and $v_r'$ and add an edge between both pairs $(v_l,v_l')$ and $(v_r,v_r')$. Continue this until all bf backedges are removed and replaced.

	\begin{remark} \label{rem:forestlabel}
		This algorithm can also be used to mark where the new leafs occur within a breadth-first exploration of the forest $F^\BF(\dd^n)$. When we first find backedge between half-edges $l$ and $r$ in $M^\BF(\dd^n)$ we replace it with two new leafs. Then as we are exploring the half-edge $l$, we find a new leaf and do not ``see'' the half-edge $r$. This means we do not kill that half-edge in step (6). This means we will eventually choose half-edge $r$ in step (2) and find second new leaf for this bf backedge. We can then label these vertices $u_1^\BF,\dotsm, u_p^\BF$ for some $p\ge n$. See Figure \ref{fig:bfforestlabeling} for an example of how this is done for the breadth-first construction.
	\end{remark}

	\begin{figure}[h!]
		\centering
		\includegraphics[width=0.9\linewidth]{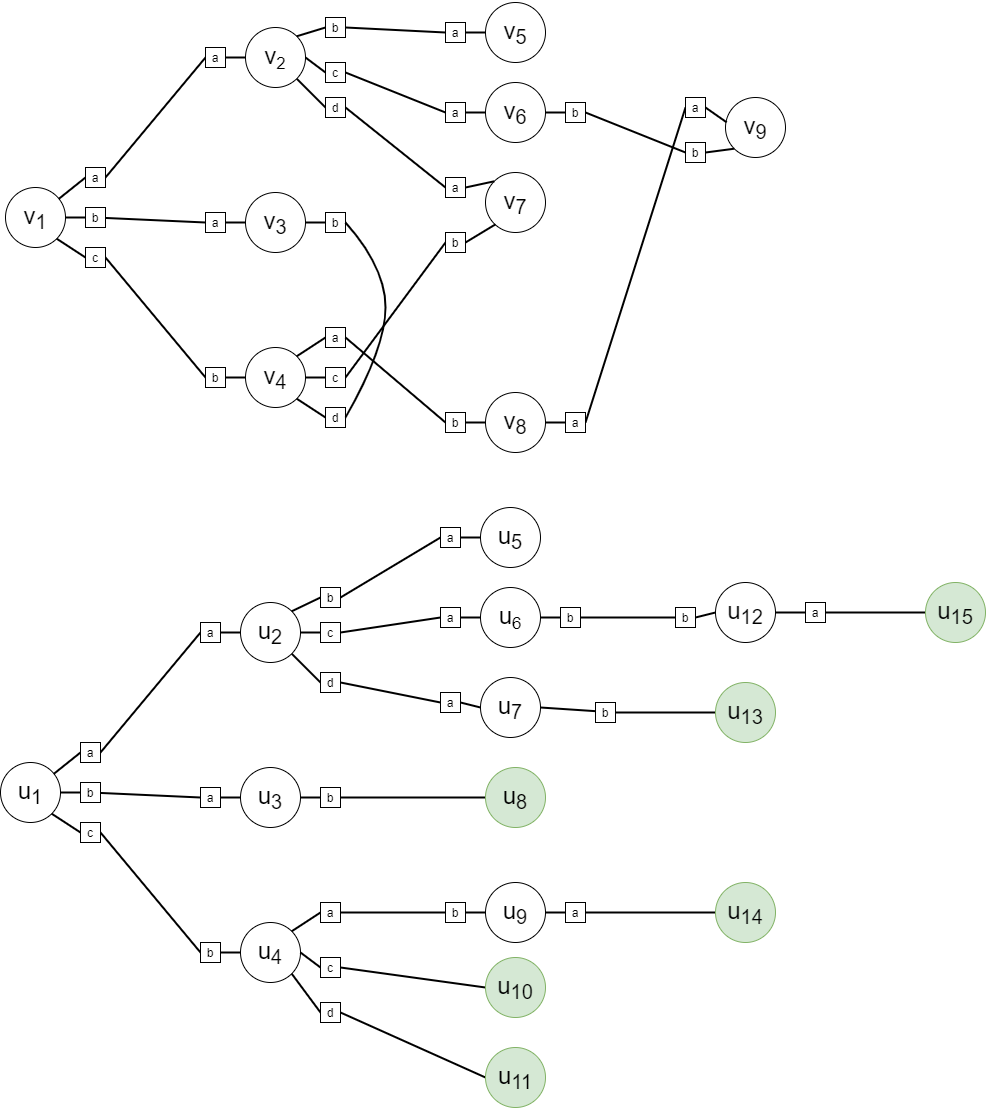}
		\caption{The first component of $M^\BF(\dd^n)$ (top) and the corresponding first component of $F^{\BF}(\dd^n)$ (bottom). The circles are the vertices, the labeled squares are the ordering of the half-edges connected to each hub with $a<b<c<d$. The three bf backedges in this graph connect half-edge $(v_3,b)$ to $(v_4,d)$, half-edge $(v_4, c)$ to $(v_7,b)$, and half-edge $(v_8,a)$ to $(v_9,a)$. The new-leafs are vertices $u_8, u_{10}, u_{11}, u_{13}, u_{14}, u_{15}$ in green and they are ordered according to the ordering of the half-edge to which it is connected.}
		\label{fig:bfforestlabeling}
	\end{figure}

	\subsubsection{Depth-first construction}
	
	We construct a graph $M^\DF(\dd^n)$ (we initially include a $\DF$ to specify the construction) as follows:
	
	We initialize at step 1 as before, we pick a sleeping half-edge uniformly at random. Label the corresponding vertex as $v_1$ and declare all of the half-edges attached to $v_1$ as active. The only thing that changes on subsequent steps is that in Case 1 we replace part (1) with
	\begin{enumerate}
		\item[(1')] Let $i$ be the \textbf{largest} integer $k$ such that there exists an active half-edge attached to $v_k$.
	\end{enumerate}
	
	This above is the depth-first construction of the multigraph $M^\DF(\dd^n)$. This changes the order in which we find vertices and label them, so we will denote the new ordering and labeling by $v_1^\DF,\dotsm, v_n^\DF$. We analogously construct the depth-first forest, by removing depth-first (df) backedges and replacing them with leaves.
	
	\subsubsection{Symmetry between constructions}
	
	Recall that the multigraph $M(\dd^n)$ is taken to be uniform over all possible pairings of half-edges. The following claim is trivial.
	\begin{claim}
		For any degree sequence $\dd^n  = (d_1,\dotsm, d_n)$ the graphs $M^\BF(\dd^n)$ and $M^\DF(\dd^n)$ are equal in distribution. We write both as $M(\dd^n)$.
	\end{claim}
	
	The symmetry between the constructions allow us to look at two random walks which turn out being equal in distribution. We write $\deg(v)$ for the degree (counted with multiplicity) of the vertex $v$ in a graph, $M(\dd^n), F^\DF(\dd^n),$ etc., which is clear from context. Recall that $(v_j^\BF;j\in[n])$ and $(v_j^\DF;j\in[n])$ are the vertices in the multigraph $M(\dd^n)$ labeled in two distinct ways acording the breadth-first exploration or depth-first explorations respectively. Define the following two walks:
	\begin{equation}\label{eqn:swalks}
		S_{\dd^n}^\BF(k) = \sum_{j=1}^k (\deg (v^\BF_j) - 2) ,\qquad S_{\dd^n}^\DF(k) = \sum_{j=1}^k (\deg(v^\DF_j) - 2).
	\end{equation}

	We wish to define analogous walks on the forests $F^{\BF}(\dd^n)$ and $F^\DF(\dd^n)$, to which we remind the reader of Remark \ref{rem:forestlabel}. Hence we have a labeling all the vertices of the forests as $(u_j^\BF)$ for the forest $F^\BF(\dd^n)$ with the breadth-first exploration and $(u_j^\DF)$ for the forest $F^\DF(\dd^n)$ with the depth-first exploration.

	For each connected component of the multigraph and hence forest, there is a vertex discovered when the collection of active vertices $\mathcal{A}$ was empty, we call those vertices \textit{roots}. If $u$ is a root in a forest, write $\chi(u) = \deg(u)$, otherwise write $\chi(u) = \deg(u) - 1$. The value of $\chi(u)$ is precisely the number of children that vertex $u$ has in the forest in which it lives. For $j$ sufficiently large, there will be no vertex of $u_j^\BF$ or $u_j^\DF$. This will not matter for our scaling limits, but for completeness we define $u_j^\BF$ and $u_j^\DF$ as root vertices of components with a single vertex, and therefore $\chi(u_j^\BF) = \chi(u_j^\DF) = 0$ for sufficiently large $j$. Define the walks
	\begin{equation}\label{eqn:xwalks1}
		X^\BF_{\dd^n}(k) = \sum_{j=1}^k (\chi(u^\BF_j) -1),\qquad X^\DF_{\dd^n}(k) = \sum_{j=1}^k (\chi(u^\DF_j)-1).
	\end{equation}
	
	As is shown in Section 5.1 of \cite{CKG.20}, the distribution of the depth first walk $X^\DF_{\dd^n}$ can be reconstructed from $S_{\dd^n}^\DF$. This was done for when the degree sequence is taken to be random; however it works for a deterministic degree sequence as well. A trivial alteration of that algorithm can be used to construct $X^{\BF}_{\dd^n}$ from the walk $S_{\dd^n}^\BF$ and moreover, we can couple these two constructions to see that backedges have a particular correspondence. We summarize this construction later in the appendix. We write this as the following lemma.
	\begin{lem}\label{lem:convBFDFequal}
		\begin{enumerate}
			\item For any degree sequence $\dd^n$. The breadth-first walks and the depth-first walks are equal in distribution. That is
			\begin{equation*}
				(S_{\dd^n}^\BF(k);k\ge 0) \overset{d}{=} (S_{\dd^n}^\DF(k); k\ge 0) ,\qquad\text{and}\qquad (X_{\dd^n}^\BF(k);k\ge 0) \overset{d}{=} (X_{\dd^n}^\DF(k); k\ge 0).
			\end{equation*}
			
			\item There exists a coupling of $F^\DF(\dd^n)$ and $F^\BF(\dd^n)$ such that for all $j<i$ a df-backedge appears between $u_j^\DF$ and $u_i^\DF$ if and only if a bf-backedge appears between $u_j^\BF$ and $u_i^\BF$.
		\end{enumerate}
		
		In particular, there exists a coupling of $X_n^\BF$ and $X_n^\DF$ such that
		\begin{equation*}
			\inf_{i\le k} X_n^\BF(i) = \inf_{i\le k} X_n^\DF(i).
		\end{equation*}
	\end{lem}

	The last part of the above lemma tells us something that will be used several times in the sequel, under this coupling the distribution excursions of $X_{\dd^n}^\BF$ and $X_{\dd^n}^\DF$ above the running infimum have the same length. 
	
	\subsubsection{Random degree distribution}
	
	In this subsection we describe some of what happens when we take the sequence $\dd^n = (d_1,\dotsm,d_n)$ to be from i.i.d. samples from a distribution $\nu$ on $\N = \{1,2,\dotsm\}$. We take $(d_j;j\ge 1)$ to be an i.i.d. sequence with common distribution $\nu$. In order to guarantee that a multigraph with degree sequence $\dd^n = (d_1,\dotsm,d_n)$ exists, we replace $d_n$ with $d_{n}+1$ if the sum has the wrong parity.
	
	Recall from the introduction that we write $M_n(\nu)$ for the random graphs with a random degree distribution and forest. We do the same when referencing the forests, i.e. we will write $F^\BF_n(\nu)$ instead of $F^\BF(\dd^n)$ when the degree sequence is random. We do not emphasize this dependence on $\nu$ when describing the random walks, instead we replace subscript $\dd^n$ with just $n$, i.e. we write $X_n^\DF$ instead of $X_{\dd^n}^\DF$. If $\nu$ has finite variance, then it can be shown (see \cite[Section 7.6]{vanderHofstad.17}) there is positive probability that the multigraph is simple, i.e. contains no self-loops nor multiple edges, and it has an explicit asymptotic formula
	\begin{equation}\label{eqn:thetaDef}
		\lim_{n\to\infty} \P(M_n(\nu) \text{ is simple}) = \exp\left(-\frac{\theta}{2} - \frac{\theta^2}{4} \right), \qquad\text{where} \quad \theta = \theta(\nu) = \frac{\E[d_1(d_1-1)]}{\E[d_1]}.
	\end{equation} Moreover, when $M_n(\nu)$ is simple, it is uniformly distributed over all simple graphs with degree distribution $\dd^n$. We write $G_n(\nu)$ for the graph $M_n(\nu)$ conditioned on it being simple.
	
	Under certain conditions on $\nu$, the value of $\theta$ also tells us something about the critical behavior of the graph which dates back to Molloy and Reed's work \cite{MR.95,MR.98}. See also \cite{JL.09}. That is if $V_n$ denotes the size of the largest component of the multigraph $M_n(\nu)$ there is a phase-transition which occurs:
	\begin{enumerate}
		\item If $\theta(\nu)>1$ then $V_n/n \overset{p}{\to} a(\nu)$ for a deterministic constant $a(\nu)>0$.
		\item If $\theta(\nu)\le 1$ then $V_n/n \overset{p}{\to} 0$.
	\end{enumerate}

	We will now restrict our attention to the case where $\nu$ is as in \eqref{eqn:cdelta1} and observe that the right-most assumption in \eqref{eqn:cdelta1} is $\E[d_1^2] =  2\E[d_1]$, which is equivalent to $\theta = 1$.

	In this setting Joseph \cite{Joseph.14} gives a scaling limit of a depth-first walk for the multigraph $M_n(\nu)$, which is very slightly different than what we wrote as $S_n^\DF$. That work was extended by Conchon-Kerjan and Goldschmidt in \cite{CKG.20}. We now recall the scaling limit in the latter reference. Let $\tilde{X} = (\tilde{X}(t);t\ge 0)$ and $\tilde{H} = (\tilde{H}(t);t\ge 0)$ be defined by the change of measure in \eqref{eqn:processStableTilt}. We write $X$ for a L\'{e}vy process with Laplace exponent \eqref{eqn:psiStable} and $H$ is its associated height process. We define the process $J_n = (J_n(k);k\ge 0)$
	\begin{equation}\label{eqn:SnJn}
		J_n(k) = \#\{j\in \{0,1,\dotsm,k-1\} : S_n^\DF(j) = \inf_{j\le \ell \le k} S_n(\ell)\},
	\end{equation}  which is a discretization of \eqref{eqn:heightProcess}.
	\begin{thm} [Joseph \cite{Joseph.14}, Conchon-Kerjan - Goldschmidt \cite{CKG.20}] \label{thm:alphaStableLimitLiterature}
		Fix some $\alpha\in (1,2)$. Let $\nu$ be a distribution satisfying \eqref{eqn:cdelta1} and write $A = \frac{c \Gamma(2-\alpha)}{\delta\alpha(\alpha-1)}$.
		Using the notation above, the following joint convergence holds in $\D^2$:
		\begin{equation*}
			\left(n^{-\frac{1}{\alpha+1}} S_n(\fl{n^{\frac{\alpha}{\alpha+1}} t}) , n^{-\frac{\alpha-1}{\alpha+1}} J_n(\fl{n^{\frac{\alpha}{\alpha+1}} t});t\ge 0 \right) \weakarrow \left(\tilde{X}(t),\tilde{H}(t);t\ge 0\right).
		\end{equation*}
	\end{thm}
	A similar result is obtained in \cite{BvdHavL.10} under a finite third moment condition on the measure $\nu$ where the limiting process is rescaling of Brownian motion with a different parabolic drift.
	
	Crucially for their descriptions of the limiting graphs, the authors of \cite{CKG.20} also develop the excursion theory for the process $\tilde{X}$ (which is notated as $\tilde{L}$ in that work). Proposition 3.9 in \cite{CKG.20} shows that the excursions of $\tilde{X},\tilde{H}$, conditioned on their length being exactly $x$ are distributed as $(\tilde{\e}_x^{(\delta)},\tilde{\hp}^{(\delta)}_x)$ defined in \eqref{eqn:tilting}.
	
	\subsubsection{Continuum Graph Limits}
	
	We now heuristically describe how the authors of \cite{CKG.20} obtain their metric space scaling limit. Let $H_n$ be the height process on the forest $F_n^\DF(\nu)$. That is $H_n(k)$ is the distance in $F^\DF_n(\nu)$ from vertex $u_k^\DF$ to the root in its connected component. This process $H_n$ satisfies \cite{LeGall.05}:
	\begin{equation*}
		H_n(k) = \#\{j\in \{0,\dotsm , k-1\}: X_n^\DF(j) = \inf_{j\le \ell\le k} X_n^\DF(\ell)\}.
	\end{equation*} 
	
	To examine the components of the graph $M_n(\nu)$, the authors of \cite{CKG.20} look at the collection of excursions of the processes $X_n^\DF$ and $H_n$. These are defined as follows:
	\begin{equation*}
		\begin{split}
			\widehat{X}_{n,i}^\DF(k) &= X_n^\DF (\sigma_{n}(i-1) + k) - X_n^\DF(\sigma_n(i-1)),\\
			\widehat{H}_{n,i}(k) &= H_n(\sigma(i-1)+k),
		\end{split} \qquad k = 0,\dotsm, \sigma_{n}(i)-\sigma_n(i-1)
	\end{equation*} where $\sigma_n(i) = \inf\{j: X_n^\DF(j) = -i\}$ is the first hitting time of level $-i$. The process $\widehat{X}_{n,i}^\DF$ starts at zero and is non-negative until it hits level $-1$ at time $k = \sigma_n(i)-\sigma_n(i-1)$.  The process $\widehat{H}_{n,i}$ is strictly positive for $k = 1,\dots, \sigma_n(i)-\sigma_n(i-1)-1$. We extend both of these by constancy for $k>\sigma_n(i)-\sigma_n(i-1)$. These processes encode the tree structure \cite{LeGall.05} of the $i^\text{th}$ connected component of the forest $F_n^\DF(\nu)$. By the construction of $M_n^\DF(\nu)$, this orders the components of the forest $F_n^\DF(\nu)$ in a manner size-biased by the number of edges in the component. There are further only a finite number of indexes $i$ such that $\sigma_n(i)-\sigma_n(i-1)\neq 1$ since for sufficiently large $i$ the $i^\text{th}$ component of $F_n^\DF(\nu)$ is simply an isolated vertex.
	
	To study the large components of the graph, we instead reorder the excursions by decreasing lengths with ties broken arbitrarily. Denote this new ordering by omitting the ``widehat'' notation: $(X_{n,i}^\DF;i\ge 1) = \left((X_{n,i}^\DF(k);k\ge 0);i\ge 1\right)$ and $(H_{n,i};i\ge 1) = \left((H_{n,i}(k);k\ge 0);i\ge 1\right)$.
	
	The $i^\text{th}$ excursion $X_{n,i}^\DF$ may not tell us information about the $i^\text{th}$ largest component of $M_n(\nu)$, $G_n^i$. This is because the forest $F_n^\DF(\nu)$ contains additional vertices which could change the ordering of the components. I.e. if $G_n^i$ had $10$ vertices and $0$ df backedges and component $G_n^{i+1}$ had $9$ vertices and 2 df backedges then the corresponding components in $F_n^\DF(\nu)$ will have $10$ and $13$ vertices respectively. In turn, their indices will appear in the opposite order. We also note that the excursions of the process $X_n^\DF$ do not identically correspond to the excursions of the process $S_n^\DF$ discussed previously. While this may cause some problems in the discrete, in the large $n$ limit neither of these problems are relevant as we will shortly explain.
	
	Before turning to the scaling limits, let $T_n^1$ be the largest connected component, i.e. tree, contained in $F_n^\DF(\nu)$. This is encoded by $X_{n,1}^\DF$ and $H_{n,1}$. This tree may contain df backedges and these backedges appear in pairs. For concreteness, suppose that there are $m\ge 1$ of these pairs. These can be indexed by $(l_1,r_1),\dotsm, (l_m,r_m)$. This means that the $l_i^\text{th}$ vertex explored in the depth-first exploration of the largest component of $T_n^1$ will be paired with the $r_i^\text{th}$ vertex explored in the corresponding component of $M_n(\nu)$. See Figure \ref{fig:excursionsbfwithmarks} for the analogous pairs for the breadth-first labeling of the component in Figure \ref{fig:bfforestlabeling}. We now define $\PP_{n,1}^\DF$ as the collection of points
	\begin{equation*}
		\PP_{n,1}^\DF = \left\{ \left( n^{-\frac{\alpha}{\alpha+1}} l_i, n^{-\frac{\alpha}{\alpha+1}} r_i\right): i =1,\dotsm, m \right\}.
	\end{equation*}  When there are no df backedges present, just define the set as the empty set. We call this set the set of \textit{marks}, and we can do the same thing to each of the other connected components as well to get sets $\PP_{n,i}^\DF$.
	
	\begin{figure}
		\centering
		\includegraphics[width=0.7\linewidth]{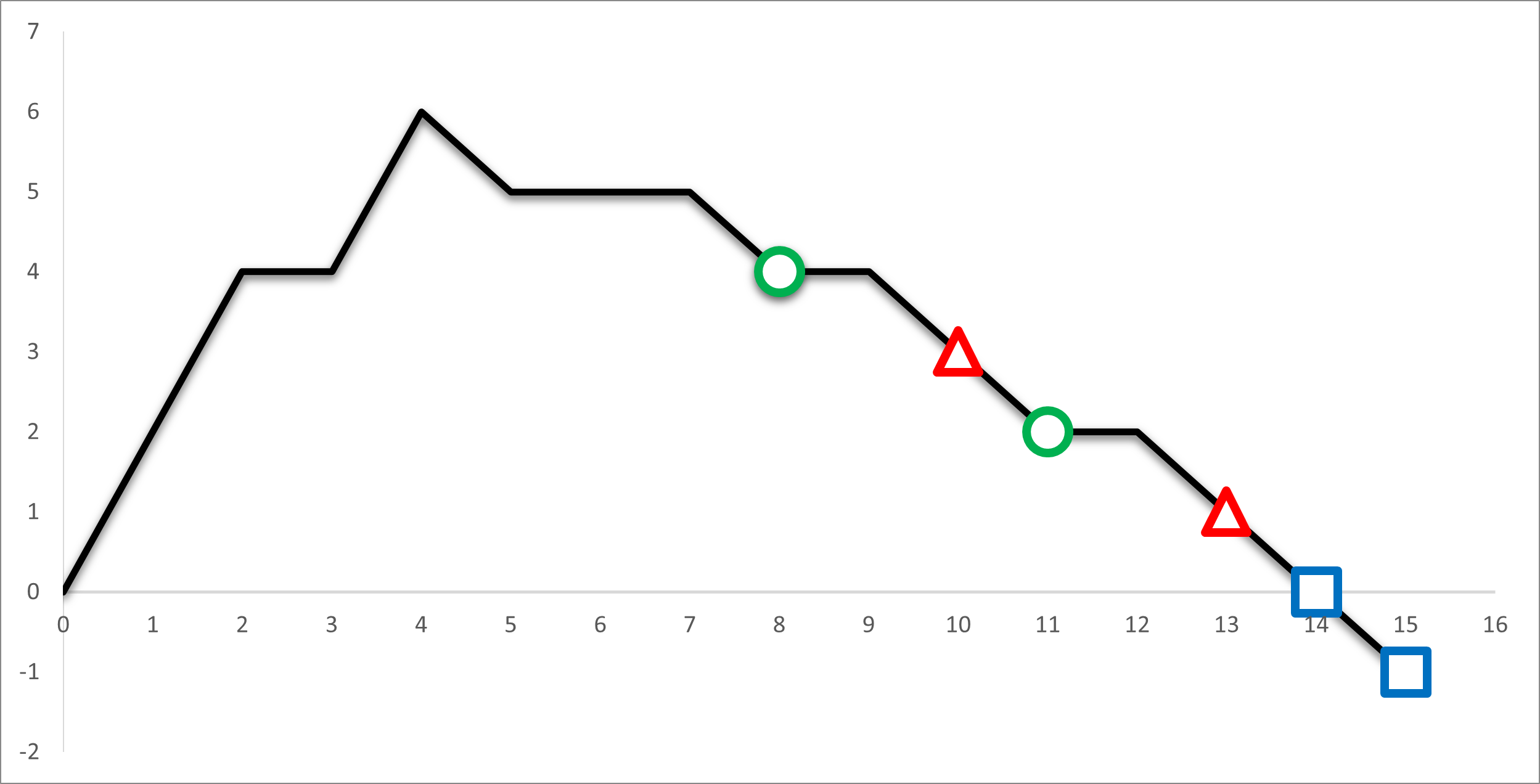}
		\caption{The excursion $X_{n,i}^\BF$ associated with the component shown in Figure \ref{fig:bfforestlabeling} with the marks $\PP^\BF_{n,i}$ included. The vertices $u_8^\BF$ and $u_{11}^\BF$ are paired, $u_{10}^\BF$ and $u_{13}^\BF$ are paired, and $u_{14}^\BF$ and $u_{15}^\BF$ are paired. These are represented by the green circles, red triangles and blue squares above.}
		\label{fig:excursionsbfwithmarks}
	\end{figure}

	An important step in how the authors of \cite{CKG.20} proved the components $G_n^1, G_n^2,\dotsm$ have scaling limits was showing scaling limits of the processes $X_{n,i}^\DF$ and $H_{n,i}$ and of the set of marks $\PP_{n,i}^\DF$. The convergence of the sets $\PP_{n,i}^\DF$ is with respect to the vague topology of its associated counting measure. Namely they prove \cite[Proposition 5.16]{CKG.20}
	\begin{equation}\label{eqn:markExcursionConv}
		\begin{split}
			&\left(\left(n^{-\frac{1}{\alpha+1}} X_{n,i}^\DF(\fl{n^{\frac{\alpha}{\alpha+1}} t});t\ge 0\right),\left( n^{-\frac{\alpha-1}{\alpha+1}} H_{n,i}(\fl{n^{\frac{\alpha}{\alpha+1} }t}); t\ge 0\right), \PP_{n,i}^\DF ;i\ge 1\right) \\
			&\qquad\qquad\qquad \weakarrow \left( \tilde{e}_i,\tilde{h}_i,\PP_i;i\ge 1\right),
		\end{split}
	\end{equation} for some discrete sets $(\PP_i;i\ge 1)$. Here the convergence in the first two coordinates is with respect to the Skorohod topology and the convergence in the third coordinate is with respect to the vague topology of its associated counting measure and then the product topology is taken over the index $i\ge 1$.

	The limiting set $\PP_i$ can be described as follows. Let $(\QQ_i;i\ge 1)$ denote an i.i.d. collection of Poisson point processes on $\R_+\times\R_+$ with intensity $\frac{1}{\delta}\Leb$. Only finitely many of these points $(s,y)\in \QQ_i$ will satisfy $0\le y\le \tilde{e}_i(s)$, and index these as $(s^1,y^1),\dotsm, (s^m,y^m)$ for some $m$. The set $\PP_i$ is then the collection
	\begin{equation*}
		\PP_i = \{(s^p,t^p): p = 1,\dotsm, m\} \qquad t^p = \inf\{u\ge s^p: \tilde{e}_i(u)\le y^p\}.
	\end{equation*}
	
	This, in turn, allowed them to show that the ordered sequence of components of $M_n(\nu)$ and $G_n(\nu)$ by conditioning converge after proper rescaling in a product \ghp topology to the sequence of continuum random graphs
	\begin{equation} \label{eqn:graphSequenceAlpha}
		\left( \M_i;i\ge 1\right) = 	\left(\GG(\tilde{h_i}, \tilde{e_i} , \QQ_i );i\ge 1\right),
	\end{equation} where $\QQ_i$ were defined above.

	Let us summarize these results in a theorem for easy reference. 
	\begin{thm} (Conchon-Kergan - Goldschmidt \cite{CKG.20})  \label{thm:ckgThm2}
		Let $(G_n^i;i\ge 1)$ denote the components of the critical random graph $M_n(\nu)$ ordered by decreasing number of vertices and viewed as pointed measured metric spaces.
		
		Under the assumptions of Theorem \ref{thm:alphaStableLimitLiterature}, the convergence in \eqref{eqn:markExcursionConv} holds in the product topology (product over the index $k$). Jointly with \eqref{eqn:markExcursionConv}, the weak convergence 
		\begin{equation}\label{eqn:graphStableconv}
			\left(\scale(n^{-\frac{\alpha-1}{\alpha+1}} , n^{-\frac{\alpha}{\alpha+1}}) G_n^i ; i \ge 1 \right) \weakarrow (\M_i; i\ge 1)
		\end{equation} holds with respect to the product \ghp topology, where the sequence $(\M_i;i\ge 1)$ is distributed as \eqref{eqn:graphSequenceAlpha}.
		
	\end{thm}
	
	\begin{remark}
		Theorem 1.1 in \cite{CKG.20} does not state the joint convergence between equations \eqref{eqn:markExcursionConv} and \eqref{eqn:graphStableconv}; however, the proof of said theorem shows there is joint convergence.
	\end{remark}
	
	These recalled results from \cite{CKG.20} are on the convergence for \textit{depth-first} objects. However, symmetry between depth-first and breadth-first constructions described above in Lemma \ref{lem:convBFDFequal} allow us to have similar results for the analogous \textit{breadth-first} object. Consequently, if we let $X_{n,i}^\BF$ be the breadth-first walk of the $i^\text{th}$ largest component of $F_n^\BF(\nu)$, or, equivalently stated, it is the $i^\text{th}$ longest excursion of $X_n^\BF$ above its running minimum, then
	\begin{equation}\label{eqn:excursionBFstable}
		\left( \left( n^{-\frac{1}{\alpha+1}} X_{n,i}^\BF (\fl{n^{\frac{\alpha}{\alpha+1}} t}) ;t\ge 0\right) ;i\ge1 \right) \weakarrow \left( \tilde{e}^*_i; i\ge 1\right),
	\end{equation} where $\left(\tilde{e}_i^*;i\ge 1\right) \overset{d}{=} \left(\tilde{e}_i;i\ge 1 \right)$. More importantly for our work, there are auxiliary processes described in \cite[pg. 30, 32]{CKG.20} and recalled in the appendix that can easily be defined in the same way for a breadth-first construction. In particular, these auxiliary processes include the collection of marks $\PP^\DF_{n,k}$ recalled above. Therefore, we can extend the convergence \eqref{eqn:excursionBFstable} by using Lemma \ref{lem:convBFDFequal} to include collection of marks in an analogous way to the depth-first marks:
	\begin{lem}\label{lem:18}
		There exists a finite set $\PP_{n,i}^\BF\subset\R_+^2$ of marks corresponding to the $i^\text{th}$ largest component of $F_n^\BF(\nu)$ which keep track of the bf backedges such that
		\begin{equation*}
			\left( X_{n,i}^\BF , \PP_{n,i}^\BF ; i \ge 1\right) \overset{d}{=} \left(X_{n,i}^\DF, \PP_{n,i}^\DF;i\ge 1 \right).
		\end{equation*}
	\end{lem}

	We now state the following lemma:
	\begin{lem} \label{lem:stableExcursionsSameSize} Couple $X_n^\BF$ and $X_n^\DF$ as in Lemma \ref{lem:convBFDFequal} so they have the same excursion intervals. Then, under the assumptions of Theorem \ref{thm:ckgThm2},
		any joint subsequential limit of \eqref{eqn:markExcursionConv}, \eqref{eqn:graphStableconv} and \eqref{eqn:excursionBFstable} satisfies for each fixed $i\ge 1$:
		\begin{equation}\label{eqn:stableExcursionsSameSize}
			\zeta(\tilde{e}_i^*) = \zeta(\tilde{e}_i)= \mu_i(\M_i),
		\end{equation} almost surely.
	\end{lem}
	\begin{proof}
		Under this conditioning, the excursion intervals of $X_n^\DF$ and $X_n^\BF$ have the same length. The equality $\zeta(\tilde{e}_i) = \mu_i(\M_i)$ follows from \cite{CKG.20}, and that implies $\zeta(\tilde{e}_i^*) = \mu_i(\M_i)$ as well. See also \cite{Joseph.14}.
	\end{proof}
	
	We have now gathered most of the required ingredients and background to prove Theorems \ref{thm:Epidemic3} and \ref{thm:ctsAlpha} using the approach in Theorem \ref{thm:conv2}. The last thing we'll verify is that Assumption 3 holds in Theorems \ref{thm:conv1} and \ref{thm:conv2}. By scaling of the $\alpha$ stable graph \cite{CKG.20}, we focus on the case that the total mass equals 1.
	
	\begin{prop}\label{prop:ass3Foralpha} Fix $\alpha\in (1,2)$.
		Let $(\ee^{(\delta)}_x,\tilde{\hp}^{(\delta)}_x)$ be defined as in \eqref{eqn:tilting}.  Let $\GG = \GG(\tilde{\hp}^{(\delta)}_x,\ee^{(\delta)}_x, \QQ)$ denote the continuum random graph where $\QQ$ is a Poisson point process with intensity $\frac{1}{\delta}\Leb$ for some $\delta>0$. Then, almost surely,
		\begin{equation*}
			\mu(B(\rho,t)\setminus\{\rho\})>0 ,\qquad\forall t>0.
		\end{equation*}
		
		The same holds for the graphs $\M_i$ appearing in \eqref{eqn:graphSequenceAlpha}.
	\end{prop}
	\begin{proof}
		
		The same statement holds for the graphs $\M_i$ hold by conditioning on their mass $\mu_i(\M_i)$ \cite[Theorem 1.2]{CKG.20}.

		We use the tree $\T_{\tilde{\hp}^{(\delta)}_x}$ as a measured metric space, and when confusion might arise we will use subscripts to specify whether we are dealing with the tree or the graph.
		
		We observe that from the quotient map
		\begin{equation*}
			q:\T_{\tilde{\hp}^{(\delta)}_x} \longrightarrow \GG(\tilde{\hp}^{(\delta)}_x,\ee^{(\delta)}_x,\QQ)
		\end{equation*} in the construction of the random graph satisfies the following:
		\begin{equation*}
			\ddd_{\GG}(\rho,q(x))\le \ddd_{\T_{\tilde{\hp}^{(\delta)}_x}} (\rho,x),\qquad \forall x\in \T_{\tilde{\hp}^{(\delta)}_x}.
		\end{equation*}
		Consequently,
		\begin{equation*}
			\mu_{\GG}\left( B_{\GG}(\rho,t)\right) \ge \mu_{\T_{\tilde{\hp}^{(\delta)}_x}}\left(B_{\T_{\tilde{\hp}^{(\delta)}_x}}(\rho,t)\right) = \int_0^1 1_{[\tilde{\hp}^{(\delta)}_x(s)\in (0,t)]}\,ds.
		\end{equation*} Since the process $\tilde{\hp}^{(\delta)}_x$ is non-negative and almost surely not identically zero. The result follows easily.
	\end{proof}
	
	\subsection{Proof of Theorem \ref{thm:Epidemic3}} \label{sec:epidemic3}

	We now prove Theorem \ref{thm:Epidemic3}.
	
	\begin{proof}[Proof of Theorem \ref{thm:Epidemic3}]
		
		Throughout the proof all limits will be as $n$ or a subsequence of $n$ goes towards infinity.
		
		Make the components of the graph $G_n^i$, as measured metric spaces with graph distance and the measure of each vertex is one. 
		
		The processes $Z_{n,i}$ measure the number of vertices infected on day $h$, which is simply the number of vertices at distance $h$ from $\rho_i$ in $G_n^i$:
		\begin{equation*}
			Z_{n,i}(h) = \#\{v\in G_n^i : d(v,\rho_i) = h\}
		\end{equation*} and the process $C_{n,i} = (C_{n,i}(h);h\ge 0)$ denote its running sum:
		\begin{equation*}
			C_{n,i}(h) = \sum_{j=0}^h Z_{n,i}(j) = \m_i(\{v\in G_n^i : d(v,\rho_i)\le h\}).
		\end{equation*} These processes measure something close to the height profile on the components of the forests $F_{n}^\BF(\nu)$; however it is not exactly the same because of the addition of new leaves. This complicates a direct application of Theorem \ref{thm:conv2}.
		
		Let us write $Z^*_{n,i}(h)$ is the discrete Lamperti transform of $X_{n,i}^\BF$:
		\begin{equation} \label{eqn:lampZstar}
			Z^*_{n,i}(h+1) = 1+ X_{n,i}^\BF\circ C_{n,i}^*(h) ,\qquad C_{n,i}^*(h) = \sum_{j=0}^h Z^*_{n,i}(j).
		\end{equation} This process $Z_{n,i}^*$ measures the number of vertices at height $h$ in the $i^\text{th}$ largest component of the forest $F_{n}^\BF(\nu)$; see the discussion around \eqref{eqn:discLamperti2} and more generally \cite{CPGUB.13}. Said another way, the values of $Z_{n,i}^*(h)$ and $Z_{n,i}(h)$ for a fixed $h$ only differ by the number of new leaves at height $h$ in the component of the forest $F_n^{\BF}(\nu)$.
		
		The total number of such additional vertices is twice the number of bf backedges (which is the number of df backedges as well). Therefore, for a fixed index $i$, the number of bf backedges in $G_n^i$ is a tight sequence in the index $n$ of random variables. Indeed, a stronger statement is true. By Proposition 5.12 in \cite{CKG.20} the weak convergence
		\begin{equation*}
			\#\{\text{bf backedges in }G_n^i\} \weakarrow \text{Poisson}\left(\frac{1}{\delta} \int_0^{\zeta(\tilde{e}_i)} \tilde{e}_i(t)\,dt \right),
		\end{equation*} where $\tilde{e}_i$ is as in \eqref{eqn:markExcursionConv}. Now for each $i$ we can bound the difference between $Z_{n,i}(h)$ and $Z_{n,i}^*(h)$ for each $i$ uniformly in $h$. Indeed
		\begin{equation*}
			\begin{split}
				\sum_{h\ge 0} |Z_{n,i}(h) - Z_{n,i}^*(h)| &=	\sum_{h\ge 0} \left|\#\{v\in G_n^i : d(v,\rho_i) = h\} - Z^*_{n,i}(h) \right|\\
				&= 2\cdot \#\{\text{bf backedges in }G_n^i\}\le \kappa_{n,i}
			\end{split}
		\end{equation*} where for each $i\ge 1$ the sequence $(\kappa_{n,i};n\ge 1)$ is a tight sequence of random variables.
		
		By Slutsky's theorem, in order to prove the rescaled convergence of $(Z_{n,i};i\ge 1)$ to the desired limit, we just need to prove the convergence of $(Z^*_{n,i};i\ge 1)$ under the same scaling regime to the same limiting processes. Indeed, their difference, when rescaled by $n^{-\frac{1}{\alpha+1}}$ converges in probability to the zero path in the Skorohod space:
		\begin{equation*}
			\left(n^{-\frac{1}{\alpha+1}} \kappa_{n,i} ;t \ge 0\right) \weakarrow \left(0; t\ge 0\right)\qquad\text{as }n\to\infty.
		\end{equation*}

		By Proposition \ref{prop:jointCX}, for any fixed integer $N$
		\begin{equation*}
			\left(\left(\left(\left(n^{-\frac{\alpha}{\alpha+1}}C^*_{n,i}(\fl{n^{\frac{\alpha-1}{\alpha+1}} t});t\ge 0\right),\left( n^{-\frac{1}{\alpha+1}} X_{n,i}^\BF (\fl{n^{\frac{\alpha}{\alpha+1}}t});t\ge 0\right) \right); i\in[N]\right) ;n\ge 1\right)
		\end{equation*} viewed as a sequence in $\D(\R_+,\R)^{2N}$, is tight. Consequently, in the product topology over the index $i\ge 1$ the sequence
		\begin{equation}\label{eqn:cxcombo}
			\left(\left(\left(\left(n^{-\frac{\alpha}{\alpha+1}}C^*_{n,i}(\fl{n^{\frac{\alpha-1}{\alpha+1}} t});t\ge 0\right),\left( n^{-\frac{1}{\alpha+1}} X_{n,i}^\BF (\fl{n^{\frac{\alpha}{\alpha+1}}t});t\ge 0\right) \right); i\ge 1\right) ;n\ge 1\right)
		\end{equation} is tight in $(\D^2)^\infty$. Additionally, by Proposition \ref{prop:jointCX} any subsequential limit, say
		\begin{equation} \label{eqn:lampLim1}
			\left(\left(\left(C_i(t);t\ge 0\right), \left(X_i(t);t\ge 0\right)\right); i\ge 1\right),
		\end{equation} must satisfy $C_i(t) = \int_0^t X_i\circ C_i(s)\,ds$. Moreover, the sequence $(X_i;i\ge 1)\overset{d}{=}(\tilde{e}_i^*;i\ge 1) \overset{d}{=}(\tilde{e}_i;i\ge 1)$. In particular subsequential limits of \eqref{eqn:cxcombo} are classified by a time-shift as in Proposition \ref{prop:AUB}.

		By Theorem \ref{thm:ckgThm2} and Lemma \ref{lem:convBFDFequal}, we know that the convergences in \eqref{eqn:markExcursionConv}, \eqref{eqn:graphStableconv} and \eqref{eqn:excursionBFstable} hold. By a tightness argument, we can assume that sequence converge jointly along a subsequence, which we will denote by the index $n$.

		Observe the sequence
		\begin{equation*}
			\left(\left(\left( n^{-\frac{\alpha}{\alpha+1}} C_{n,i}(\fl{n^{\frac{\alpha-1}{\alpha+1}}t}); t\ge 0\right);i\ge 1\right);n\ge 1\right)
		\end{equation*} is tight in $\D^\infty$. Indeed, this easily follows the tightness of
		\begin{equation*}
			\left(\left(	\left( n^{-\frac{\alpha}{\alpha+1}} C_{n,i}^*(\fl{n^{\frac{\alpha-1}{\alpha+1}}t}); t\ge 0\right);i\ge1\right);n\ge 1\right)
		\end{equation*} in $\D^\infty$ discussed above
		and the bounds
		\begin{equation*}
			|C_{n,i}(h) - C_{n,i}^*(h)| \le \sum_{j\ge 0} |Z_{n,i}(j)-Z_{n,i}^*(j)| \le \kappa_{n,i}. 
		\end{equation*} 
		
		Let us work on a subsequence of \eqref{eqn:cxcombo} which converges to \eqref{eqn:lampLim1}. Call this index $n_j$. Then, by the previous paragraph, 
		\begin{equation*}
			\left(\left(	n_j^{-\frac{\alpha}{\alpha+1}} C_{n_j,i}( \fl{n_j^{\frac{\alpha-1}{\alpha+1}}t});t\ge 0\right) ;i\ge 1\right) \weakarrow \left(\left({C}_i(t);t\ge 0\right);i\ge 1\right)
		\end{equation*} for the same processes ${C}_i$ in \eqref{eqn:lampLim1}. However, $C_{n,i}(h) = \#\{v\in G_{n}^i: d(v,\rho_i)\le h\}$ is just the measure of the ball of radius $h$ in $G_n^i$ and so an application of Lemma \ref{lem:ballCont} implies
		\begin{equation}\label{eqn:cConvtoMeas}
			{n_j}^{-\frac{\alpha}{\alpha+1}} C_{n_j,i}(\fl{n_j^{\frac{\alpha-1}{\alpha+1}}t}) \weakarrow \mu_i(B(\rho,t)), \qquad \text{for all but countably many  }t>0,
		\end{equation} where $\mu_i$ is the mass measure on the scaling limit of the graph component $G_n^i$. Hence ${C}_i(t)$ must satisfy:
		\begin{equation*}
			{C}_i(t) \overset{d}{=}\mu_i(B(\rho,t)), \qquad\text{for Lebsgue a.e. } t>0.
		\end{equation*} It follows easily from Proposition \ref{prop:ass3Foralpha} that
		\begin{equation*}
			\P(C_i(t)>0, \forall t>0)  =1,\qquad \forall i\ge 1.
		\end{equation*} Indeed $C_i(t)$ is non-decreasing, and we can find a countable dense set of $t>0$ such that $P(C_i(t) > 0) = 1$. Hence, by Propositions \ref{prop:AUB} and $(X_i;i\ge 1) \overset{d}{=} (\tilde{e}_i;i\ge 1)$, we get $((C_i, X_i);i\ge 1) \overset{d}{=} ((C_i,\tilde{e}_i);i\ge 1)$ where $(Z_i,C_i)$ is the Lamperti pair associated with $\tilde{e}_i$. Since this works for any subsequential limit, we conclude that the original sequence converges:
		\begin{equation*}
			\left(	\left(n^{-\frac{\alpha}{\alpha+1}} C_{n,i}(\fl{n^{\frac{\alpha-1}{\alpha+1} } t};t\ge 0 \right);i\ge 1\right) \weakarrow (C_i;i\ge 1).
		\end{equation*}

		A similar proof of Proposition \ref{prop:jointCXwithLimit} yields the joint convergence
		\begin{equation}\label{eqn:t1}
			\left(\left( n^{-\frac{1}{\alpha+1}} Z^*_{n,i}(\fl{n^{\frac{\alpha-1}{\alpha+1}}t});t\ge 0\right), \left( n^{-\frac{\alpha}{\alpha+1}} C_{n,i}^*(\fl{n^{\frac{\alpha-1}{\alpha+1}}t}); t\ge 0\right) ;i\ge 1\right)\weakarrow  ((Z_i,C_i);i\ge 1)
		\end{equation} where $(Z_i,C_i)$ is the Lamperti pair associated with the excursion $\tilde{e}_i$.
		
		This proves the desired claim.
	\end{proof}

	Before turning to proof of Theorem \ref{thm:ctsAlpha}, we state and prove the following lemma:
	\begin{lem}\label{lem:20} Couple the depth-first and breadth-first walks as in Lemma \ref{lem:convBFDFequal}.
		Under the assumptions of Theorem \ref{thm:alphaStableLimitLiterature}, and using the notation in \eqref{eqn:lampZstar}. There is joint convergence in distribution along a subsequence of the index $n\ge 1$ of the collection
		\begin{equation*}
			\begin{split}
				\bigg(	\bigg( \Big( n^{-\frac{1}{\alpha+1}}& Z^*_{n,i}(\fl{n^{\frac{\alpha-1}{\alpha+1}}t});t\ge 0\Big),\Big( n^{-\frac{\alpha}{\alpha+1}} C_{n,i}^*(\fl{n^{\frac{\alpha-1}{\alpha+1}}t});t\ge 0\Big),\\
				& \Big(n^{-\frac{1}{\alpha+1}} X_{n,i}^\BF (\fl{n^{\frac{\alpha}{\alpha+1}} t}) ;t\ge 0\Big), \scale(n^{-\frac{\alpha-1}{\alpha+1}}, n^{-\frac{\alpha}{\alpha+1}} )G_n^i, \#\PP_{n,i}^\BF \bigg);i\ge 1\bigg)_{n\ge 1}
			\end{split}
		\end{equation*} towards
		\begin{equation*}
			((Z_i,C_i,\tilde{e}^*_i, \M_i, \sur(\M_i));i\ge 1),
		\end{equation*} where
		\begin{enumerate}
			\item $\M_i$ are as in Theorem \ref{thm:ckgThm2} and \eqref{eqn:graphSequenceAlpha};
			\item $(Z_i,C_i)$ is the Lamperti pair associated with the excursion $\tilde{e}^*_i$;
			\item The process $C_i(t) = \mu_i(B(\rho,t))$ for almost all (and hence all) $t\ge 0$;
			\item The excursions $(\tilde{e}_i^*;i\ge 1)\overset{d}{=} (\tilde{e}_i;i\ge 1)$ in the construction of $\M_i$, and, in particular, the length of the excursion $\tilde{e}_i^*$ is the mass of the space $\M_i$, i.e.
			\begin{equation*}
				\zeta(\tilde{e}_i^*) = \M_i;
			\end{equation*}
			\item The random variable $\sur(\M_i)$, the surplus of the space $\M_i$ is
			\begin{equation*}
				\sur(\M_i) \overset{d}{=} \operatorname{Poisson}\left(\frac{1}{\delta}\int_0^{\zeta(\tilde{e}_i^*)}\tilde{e}_i^*(t)\,dt\right);
			\end{equation*}
			\item Lastly, conditionally on the length of excursion $\zeta_i:=\zeta(\tilde{e}_i^*)$ and the surplus values $R_i:=\sur(\M_i)$, the graph $\M_i$ satisfies
			\begin{equation*}
				\M_i \overset{d}{=} \scale\left(\zeta_i^{(\alpha-1)/\alpha} ,\zeta_i \right) \GG^{(\alpha,R_i)}
			\end{equation*}
		\end{enumerate}
	\end{lem}
	
	\begin{proof}
		These are tight random variables in each of the marginals, so joint convergence along a subsequence is standard.
		
		Item 1 follows from the referenced theorem.
		
		Item 2 follows from the proof of Theorem \ref{thm:Epidemic3} and the identity in distribution $(\tilde{e}_i;i\ge 1)\overset{d}{=}( \tilde{e}_i^*;i\ge 1)$ previously seen.
		
		Item 3 follows from the proof of Theorem \ref{thm:Epidemic3}, particularly around \eqref{eqn:cConvtoMeas}.
		
		Item 4 is from Lemma \ref{lem:stableExcursionsSameSize}.
		
		Item 5 follows from Theorem 5.5 and Proposition 5.12 in \cite{CKG.20} along with the equality $\#\PP_{n,1}^\DF = \#\PP_{n,i}^\BF$.
		
		Item 6 follows from the proof of Theorem 1.2 in \cite{CKG.20}.
	\end{proof}

	\subsection{Proof of Theorem \ref{thm:ctsAlpha}} \label{sec:ctsAlpha}

	\begin{proof}[Proof of Theorem \ref{thm:ctsAlpha}]

		The big content of this proof is to show that we can condition on the length $\zeta(\tilde{e}_i^*)$ of the excursion $\tilde{e}_i^*$ and the surplus of the graphs $\M_i$ by using Lemma \ref{lem:20} and scaling results for the excursions proved in \cite{CKG.20}.

		By Lemma \ref{lem:20}, we know that we can write the height profile of the graph $\M_i$ (which is of random mass) as the process $Z_i$ where $(Z_i,C_i)$ is the Lamperti pair associated with the excursion $\tilde{e}^*_i$. In fact, we know
		\begin{equation*}
			\left( (\mu_i(B(\rho_i,v);v\ge 0), \mu_i(\M_i), \sur(\M_i)\right)_{i\ge 1} \overset{d}{=} \left((C_i(v);v\ge 0), \zeta(\tilde{e}_i^*), R_i \right)_{i\ge 1}
		\end{equation*} where $(Z_i,C_i)$ is the Lamperti pair associated with the excursion $\tilde{e}_i^*$ and \linebreak $R_i\sim$ $ \text{Poisson}\left(\frac{1}{\delta}\int_0^{\zeta(\tilde{e}_i^*)} \tilde{e}^*_i(t)\,dt\right)$.
		
		Conditioning on the values of $\zeta(\tilde{e}_1^*)$ and $R_1$ gives
		\begin{equation} \label{eqn:conditEq}
			\left( (\mu_1(B(\rho_1,v));v\ge 0)\bigg| \mu_1(\M_1) = 1, \sur(\M_1) = k\right) \overset{d}{=} \left((C_1(t);t\ge 0) \bigg| \zeta(\tilde{e}_1^*) = 1, R_1 = k \right).
		\end{equation}

		We can use the \textit{proof} of Theorem 1.2 in \cite{CKG.20} to handle this conditioning on the right-hand side  and the statement of Theorem 1.2 in \cite{CKG.20} to handle the left-hand side.
		
		The conditioning in the proof of Theorem 1.2 in \cite{CKG.20} gives
		\begin{equation*}
			\E\left[ g(\tilde{e}_i^*)\big| R_i = k, \zeta(\tilde{e}_i) = 1\right] = \E[g(\e^{(k)})]
		\end{equation*} for all positive functionals $g$ and where $\e^{(k)} = (e^{(k)}(t);t\in[0,1])$ is defined in \eqref{eqn:ktilt}. In particular this holds for $i = 1$. Recall from Section \ref{sec:lamperti}, that $c_1$ is simply a functional of $\tilde{e}_1^*$. Therefore, conditionally on the values of $\zeta(\tilde{e}^*_1)$ and $R_1$ we have
		\begin{equation*}
			\left((C_1(t);t\ge 0) \bigg| \zeta(\tilde{e}_1^*) = 1, R_1 = k \right) \overset{d}{=} \left(\ccc^{(k)}(t);t\ge 0 \right)
		\end{equation*} where $(\z^{(k)},\ccc^{(k)})$ is the Lamperti pair associated with the excursion $\e^{(k)}$. 
		
		The left-hand side of \eqref{eqn:conditEq} is easy to condition with part (6) of Lemma \ref{lem:20}. Conditionally on the values of $\mu_1(\M_1)$ and $\sur(\M_1)$ (which is precisely the conditioning described above for $c_1$) the metric spaces $\M_1$ satisfies
		\begin{equation*}
			\left(\M_1 \bigg| \mu_1(\M_1) = 1, \sur(\M_1) = k\right) \overset{d}{=}  \GG^{(\alpha,k)}.
		\end{equation*}
		Hence
		\begin{equation*}
			\left(\mu_{\GG^{(\alpha,k)}}(B(\rho,v));v\ge 0 \right) \overset{d}{=} \left(\ccc^{(k)}(v);v\ge 0 \right).
		\end{equation*} An application of Proposition \ref{prop:AUB} completes the proof.
	\end{proof}
	
	\section{Discussion} \label{sec:disc}

	In this work we showed convergence of the height profiles for the macroscopic components of a certain class of critical random graphs. We did this by looking at the height profile of these graphs and we relied on the weak convergence results that exist in the literature on some encoding stochastic processes. We observe that these techniques can likely be extended to other graph models appearing in the literature.
	
	For example, the work of Broutin, Duquesne and Wang \cite{BDW.18,BDW.20} provides the rescaled convergence under certain conditions of the \textit{rank-1 inhomogeneous model} associated to a weight sequence $\ww = (w_1,\dotsm,w_n)$. That graph, whose asymptotics were studied by Aldous and Limic in \cite{AL.98}, is a graph on $n$ vertices where edges are added independently with probability
	\begin{equation*}
		\P\left( \text{edge} \{i,j\} \text{ is included}\right)= 1 - \exp\left(-w_i w_j /q\right)
	\end{equation*} for some parameter $q>0$. This graph goes by other names as well: the Poisson random graph \cite{BvdHavL.10,NR.06} and the Norros-Reittu model \cite{BvdHavL.10}. See also \cite{BvdHRS.18,BvdHavL.12} and Section 6.8.2 of \cite{vanderHofstad.17} for more information.
	The resulting limiting processes and graphs are related to L\'{e}vy-type processes (sometimes called L\'{e}vy processes without replacement) constructed from spectrally positive L\'evy processes which {are not} stable. As in \cite{ABBG.12, CKG.20}, Broutin, Duquesne and Wang show convergence of the graphs as metric spaces by using a depth-first descriptions. However, there is also convergence of the breadth-first walks \cite{AL.98}, and so proving convergence of the height profiles should similar to the proof of Theorems \ref{thm:Epidemic3} and Theorems \ref{thm:conv2}.
	
	Using the results in the literature on Galton-Watson trees conditioned on having a fixed size \cite{LeGall.05,Duquesne.03, MM.03, Aldous.93} one can recover the Jeulin identity \cite{JY.85} and its $\alpha$-stable extension due to Miermont \cite{Miermont.03} from our Theorem \ref{thm:conv1} as well. The proofs in \cite{Miermont.03, JY.85} do not rely on weak convergence arguments. For proofs using weak-convergence arguments more in-line with the results of this papers see Kersting's work \cite{Kersting.11}, or joint work of Angtuncio and Uribe Bravo \cite{AUB.20}. See also \cite{AMP.04} for a weak convergence result in a slightly weaker topology.
	
	More generally, under certain conditions (see Theorem 2.3.1 in \cite{DL.02}) on the offspring distribution, there is convergence of Galton-Watson forests to continuum forests encoded by spectrally positive L\'{e}vy processes. Under these assumptions, one can use a modification of Lemma 4.8 in \cite{MR.17} or Lemma 5.8 in \cite{CKG.20}, one should be able to prove a Jeulin-type identity for excursions for non-stable L\'{e}vy processes and their associated height processes by a simple application of Theorem \ref{thm:conv2}. As far as the author is aware, such results are not present in the literature.

	\begin{appendix} 
		
		\section*{} 
		
		In this appendix we recall the construction of the depth-first walks and auxiliary processes in an analogous way that Conchon-Kerjan and Goldschmidt \cite[pg. 30, 32]{CKG.20} do in their work. We do this from a deterministic degree sequence, whereas they work with a random degree sequence. That is we fix an $n$ and $\dd^n = (d_1,\dotsm, d_n)$ with $d_j\ge 1$. We omit reference to $\dd^n$ from our notation.
		
		We let $M$ denote a uniformly random multigraph with degree distribution $\dd^n$. The vertices of $M$ can be ordered in a depth-first order, $v_1^\DF,\dotsm, v_n^\DF$, or a breadth-first order $v_1^\BF,\dotsm, v_n^\BF$. We let $D_j^\DF = \deg(v^\DF_j)$ and $D_j^\BF = \deg(v^\BF_j)$ counted with multiplicity. Let $F^\DF$ and $F^\BF$ be the forests constructed from $M$ by removing backedges and replacing them with two leaves and let $X^\DF$ (resp. $X^\BF$) denote the depth-first (resp. component-by-component breadth-first) walk on $F^\DF$ (resp. $F^\BF$).
		
		We write
		\begin{equation*}
			S^\DF(k) = \sum_{j=1}^k( D^\DF_j - 2),\qquad S^\BF(k) = \sum_{j=1}^k (D_j^\BF - 2).
		\end{equation*} These walks appeared (with slightly different notation) in \eqref{eqn:swalks} however they are equal in distribution.

		Let us now discuss the construction of the walk $X^\DF$ in \eqref{eqn:xwalks1} from the walk $S^\DF$. We start with $N^\DF(0) = X^\DF(0) = 0$, $\M^\DF(0) = \emptyset$ and $\tau^\DF(0) = 0$. The process $N^\DF$ counts the number of df backedges discovered in the graph at step $k$, and the set-valued process $\M^\DF$ keeps track of the marks corresponding to the df backedges. The process $\tau^\DF(k)$ is a time-change relating to the new leaves that will be included in the forest $F^\DF$. For $k\ge 0$:
		\begin{itemize}
			\item \textbf{New component of $F^\DF$ is discovered}\\If $X^\DF(k) = \min_{0\le i\le k-1} X^\DF(i) - 1$ or $k = 0$, then we have discovered a new component. We set $\tau^\DF(k+1) = \tau^\DF(k)+1$, $N^\DF(k+1) = N^\DF(k)$, $\M^\DF(k+1) = \M^\DF(k)$, and
			\begin{equation*}
				X^\DF(k+1) = X^\DF(k) + S^\DF(\tau^\DF(k)+1) - S^\DF(\tau^\DF(k)) +1.
			\end{equation*}
			\item \textbf{Determine if we start a back-edge or not}\\
			If $X^\DF(k)> \min_{0\le i\le k-1} X^\DF(i) -1$ and $X^\DF(k)\notin \M^\DF(k)$ then the vertex $u_{k+1}^\DF$ in $F^\DF$ is not a new-leaf paired to a previously explored new-leaf. There is still a chance that $u_{k+1}^\DF$ is a new-leaf which is paired with an undiscovered new leaf.
			\begin{itemize}
				\item \textbf{The vertex $u_{k+1}^\DF$ is a new-leaf}\\
				The vertex $u_{k+1}^\DF$ is a new leaf with probability
				\begin{equation*}
					\frac{\displaystyle X^\DF(k) - \min_{0\le i\le k} X^\DF(i) - \#\M^\DF(k)}{ \displaystyle X^\DF(k) - \min_{0\le i\le k} X^\DF(i) - \# \M^\DF(k) + \sum_{j = \tau^\DF(k)+1}^n D^\DF_j}.
				\end{equation*} Above, the numerator represents the number of active half-edges in the corresponding exploration of the mulitgraph. We substract that $\#\M^\DF(k)$ term because these represent $\#\M^\DF(k)$ new-leafs which would have already been killed at the corresponding step in the exploration of the multigraph. The denominator counts the total number of half-edges yet to be explored in the multigraph.
				
				In this situation, let $\tau^\DF(k+1)  = \tau^\DF(k)$. We also now that $u_{k+1}^\DF$ is a new leaf with no children and so we set $X^\DF(k+1) = X^\DF(k)-1$. Let $N^\DF(k+1) = N^\DF(k)+1$ and sample $U^\DF(k+1)$ uniformly from
				\begin{equation*}
					\left\{ \min_{0\le i\le k} X^\DF(i), \min_{0\le i\le k} X^\DF(i) +1,\dotsm, X^\DF(k)-1\right\} \setminus \M^\DF(k).
				\end{equation*} This new leaf $u_{k+1}^\DF$ will be connected to a the vertex $U^\DF(k+1)$.
				
				Lastly, set $\M^\DF(k+1) = \M^\DF(k)\cup\{U^\DF(k+1)\}$.
				
				\item \textbf{The vertex $u_{k+1}^\DF$ is part of the original multigraph}\\ With the complement probability the vertex $u_{k+1}^\DF$ is not a new leaf. In which case set $\tau^\DF(k+1) = \tau^\DF(k)+1$,
				\begin{equation*}
					X^\DF(k+1) = X^\DF(k) + S^\DF(\tau^\DF(k) +1 ) - S^\DF(\tau^\DF(k)),
				\end{equation*} $N^\DF(k+1) = N^\DF(k)$ and $\M^\DF(k+1) = \M^\DF(k)$.
			\end{itemize}
			
			\item \textbf{Ending a back-edge}\\
			If $X^\DF(k) > \min_{0\le i\le k-1} X^\DF(i)-1$ and $X^\DF(k)\in \M^\DF(k)$ then the vertex $u_{k+1}^\DF$ is a new-leaf which is connected to a previously discovered new leaf. We let $\tau^\DF(k+1) = \tau^\DF(k)$, $X^\DF(k+1) = X^\DF(k) -1$, $N^\DF(k+1) = N^\DF(k)$ and $\M^\DF(k+1)= \M^\DF(k) \setminus\{X^\DF(k)\}$.
		\end{itemize}

		We note that the above construction is the discretized version of creating the continuum random graphs in Section \ref{sec:realGraphs}. The term marks in the main body of this work were described by $\PP^\DF_{n,k}$ in the $k^\text{th}$ component of the graph. These marks of the time-shifting and time-scalings of the pairs $(i+1,U(i+1))$ whenever $N(i+1) - N(i) = 1$.
		
		Of course, we can do the same process with replacing every $\DF$ above with a $\BF$ and see that there is a distributionally equivalent way of constructing the walk $X^\BF$ from the walk $S^\BF$. Since the breadth-first and depth-first constructed multigraphs are equal in distribution, and by the first part of Lemma \ref{lem:convBFDFequal}, which has a trivial proof, we can use the above construction to create a coupling between the forests $F^\DF$ and $F^\BF$ where the depth-first walk on the former $X^\DF$ is the same as the breadth-first walk on the latter $X^\BF$. Moreover, this shows that a df-backedge between $u_j^\DF$ and $u_i^\DF$ in this coupling corresponds to precisely the a bf-backedge between $u_j^\BF$ and $u_i^\BF$. An analogous symmetry was used in \cite{AMP.04} to describe the height profile of inhomogeneous continuum random trees. 
	\end{appendix}

	%
	%
	
	\section*{Acknowledgements}
	
	Research supported in part by NSF Grant DMS-1444084.
	
	The author would like to thank Soumik Pal for continuing guidance and support. In particular his suggestions drastically improved the clarity of the introduction. 
	
	

\end{document}